\documentclass[12pt]{article}

\usepackage[margin=1.3in, top=1.3in, bottom=1.3in]{geometry}

\usepackage[utf8]{inputenc} 
\usepackage[T1]{fontenc}    
\usepackage{url}            
\usepackage{booktabs}       
\usepackage{amsfonts}       
\usepackage{nicefrac}       
\usepackage{microtype}      

\usepackage{amsmath,amssymb,amsthm,xcolor}
\usepackage{array}
\usepackage{float}
\theoremstyle{definition}
\newtheorem{example}{Example}
\newtheorem{definition}{Definition}
\newtheorem{remark}{Remark}
\theoremstyle{plain}
\newtheorem{theorem}{Theorem}
\newtheorem{lemma}{Lemma}
\newtheorem{proposition}{Proposition}
\newtheorem{corollary}{Corollary}
\usepackage{graphicx}
\usepackage{subcaption}
\usepackage{mdframed}
\usepackage{lipsum}
\newmdtheoremenv{theo}{Theorem}
\usepackage{wrapfig}
\usepackage[hidelinks,hypertexnames=false]{hyperref}
\usepackage{tikz}
\usepackage{enumitem}
\allowdisplaybreaks
\usepackage{pgfplots}
\usetikzlibrary{intersections, pgfplots.fillbetween}
\usepackage{epstopdf}
\usepackage{algorithmic}
\usepackage{booktabs}
\usepackage{changepage}

\title{\LARGE Global convergence of the gradient method for functions definable in o-minimal structures}

\begin{document}

\author{\large C\'edric Josz\thanks{\url{cj2638@columbia.edu}, IEOR, Columbia University, New York. Research supported by NSF EPCN grant 2023032 and ONR grant N00014-21-1-2282.}}
\date{}

\maketitle

\begin{center}
    \textbf{Abstract}
    \end{center}
    \vspace*{-3mm}
 \begin{adjustwidth}{0.2in}{0.2in}
~~~~We consider the gradient method with variable step size for minimizing functions that are definable in o-minimal structures on the real field and differentiable with locally Lipschitz gradients. We prove that global convergence holds if continuous gradient trajectories are bounded, with the minimum gradient norm vanishing at the rate $o(1/k)$ if the step sizes are greater than a positive constant. If additionally the gradient is continuously differentiable, all saddle points are strict, and the step sizes are constant, then convergence to a local minimum holds almost surely over any bounded set of initial points.
\end{adjustwidth} 
\vspace*{3mm}
\noindent{\bf Keywords:} Kurdyka-\L{}ojasiewicz inequality, Morse-Sard theorem, semi-algebraic geometry.

\section{Introduction}
\label{sec:Introduction}
The gradient method for minimizing a differentiable function $f: \mathbb{R}^n\rightarrow\mathbb{R}$ consists in choosing an initial point $x_0 \in \mathbb{R}^n$ and generating a sequence of iterates according to the update rule $x_{k+1} = x_k - \alpha_k \nabla f(x_k)$ for $k=0,1,2,\hdots$ where $\alpha_0,\alpha_1,\alpha_2,\hdots>0$ are called the step sizes. Many objective functions $f$ of interest nowadays are not coercive and their gradient is not Lipschitz continuous, including those arising in nonconvex statistical estimation problems. Thus none of the convergence theorems regarding the gradient method apply, to the best of our knowledge. Merely showing that the iterates are bounded has remained elusive: it actually appears as an assumption in many recent works on the gradient method \cite[Theorem 4.1]{absil2005convergence}, \cite[Theorem 3.2]{attouch2013convergence}, \cite[Proposition 14]{lee2016}, \cite[Theorem 3]{panageas2017}, and \cite[Assumption 7]{lee2019}. Even then, one needs to either choose the step sizes carefully or assume Lipschitz continuity of the gradient in order to obtain convergence results (see Section \ref{sec:Literature review}). The object of this paper is to take a first step towards filling this gap with Theorem \ref{thm:convergence}. 
We next give some notations.

Let $\mathbb{N} := \{0,1,2,\hdots\}$ and let $\|\cdot\|$ be the induced norm of an inner product $\langle \cdot, \cdot\rangle$ on $\mathbb{R}^n$. Throughout this paper, we fix an arbitrary o-minimal structure on the real field (see Definition \ref{def:o-minimal}), and say that sets or functions are definable if they are definable in this structure. For readers not familiar with o-minimal structures, note that seemingly all objective functions used in practice are definable in some o-minimal structure. Low-rank matrix recovery (for e.g., matrix factorization, matrix completion, and matrix sensing) and deep neural networks are no exception \cite[Section 6.2]{bolte2020conservative}.

\begin{theorem}
\label{thm:convergence}
Let $f:\mathbb{R}^n \rightarrow \mathbb{R}$ be a definable differentiable function whose gradient is locally Lipschitz. The following statements are equivalent:
\begin{enumerate}
    \item For all $x_0 \in \mathbb{R}^n$, there exist $\bar{\alpha}>0$ and $c>0$ such that for all $\alpha_0,\alpha_1,\alpha_2,\hdots \in (0,\bar{\alpha}]$, the sequence $x_0,x_1,x_2,\hdots \in \mathbb{R}^n$ defined by 
    \begin{equation}
    \label{eq:gradient_method}
        x_{k+1} = x_k - \alpha_k \nabla f(x_k), ~~~\forall k \in \mathbb{N},
    \end{equation}
    satisfies $\|x_k\|\leqslant c$ for all $k \in \mathbb{N}$.
    \item For all $0< T \leqslant \infty$ and any differentiable function $x:[0,T)\rightarrow\mathbb{R}^n$ such that
    \begin{equation*}
        x'(t) = -\nabla f(x(t)), ~~~ \forall t \in (0,T),
    \end{equation*}
    there exists $c>0$ such that $\|x(t)\| \leqslant c$ for all $t\in [0,T)$.
    \item For any bounded subset $X_0$ of $\mathbb{R}^n$, there exist $\bar{\alpha}>0$ and $c>0$ such that, for all $\alpha_0,\alpha_1,\alpha_2,\hdots \in (0,\bar{\alpha}]$ and $x_0 \in X_0$, the sequence $x_0,x_1,x_2,\hdots \in \mathbb{R}^n$ defined by \eqref{eq:gradient_method} obeys 
    \begin{equation*}
        \sum_{k=0}^{\infty} \|x_{k+1}-x_k\| \leqslant c.
    \end{equation*}
\end{enumerate}
\end{theorem}

The bound on the length implies that the iterates converge to a critical point of $f$ (i.e., a point $x^* \in \mathbb{R}^n$ such that $\nabla f(x^*)=0$) if the step sizes are not summable. If in particular there exists $\underline{\alpha}>0$ such that $\alpha_0,\alpha_1,\alpha_2,\hdots \geqslant \underline{\alpha}$, then 
\begin{equation*}
    \min_{i=0,\hdots,k} \|\nabla f(x_i)\| ~\leqslant~ \frac{2\underline{\alpha}^{-1}}{k+2} \sum_{i=\lfloor k/2 \rfloor}^{\infty} \|x_{i+1}-x_i\|, ~~~ \forall k\in \mathbb{N},
\end{equation*}
where $\lfloor \cdot \rfloor$ stands for floor\footnote{Any nonnegative decreasing sequence $u_0,u_1,u_2,\hdots$, and in particular the minimum gradient norm, satisfies $(k/2+1)u_k \leqslant (k -\lfloor k/2 \rfloor +1)u_k \leqslant \sum_{i=\lfloor k/2\rfloor}^k u_i \leqslant \sum_{i=\lfloor k/2\rfloor}^\infty u_i$ for all $k\in \mathbb{N}$.}. The previously known global convergence rate is $o(1/\sqrt{k})$ for lower bounded differentiable functions with Lipschitz continuous gradients \cite[Equation (1.2.22)]{nesterov2018lectures} \cite[Lemma 1]{lee2019first}, which improves to $o(1/k)$ if additionally the function is convex and attains its infimum \cite[Equation (2)]{nesterov2012make} \cite[Theorem 3]{lee2019first}. Corollary \ref{cor:local_min} gives sufficient conditions for the critical point to be a local minimum.

\begin{corollary}
\label{cor:local_min}
Let $f:\mathbb{R}^n \rightarrow \mathbb{R}$ be a definable twice continuously differentiable function with bounded continuous gradient trajectories. If the Hessian of $f$ has a negative eigenvalue at all saddle points of $f$, then for any bounded subset $X_0$ of $\mathbb{R}^n$, there exists $\bar{\alpha}>0$ such that for all $\alpha \in (0,\bar{\alpha}]$ and for almost every $x_0 \in X_0$, the sequence $x_0,x_1,x_2,\hdots \in \mathbb{R}^n$ defined by $x_{k+1} = x_k - \alpha \nabla f(x_k)$ for all $k \in \mathbb{N}$ converges to a local minimum of $f$.
\end{corollary}

Corollary \ref{cor:local_min} immediately follows from Theorem \ref{thm:convergence} and \cite[Theorem 3]{panageas2017} via the center and stable manifolds theorem \cite[Theorem III.7]{shub2013global}. In the o-minimal setting, no assumption on the Hessian is required at local maxima, unlike in existing work \cite{lee2016,lee2019,panageas2019}. Indeed, local maxima that are not local minima lie on the boundary of the set of critical points, which has measure zero by \cite[(1.10) Corollary p. 68]{van1998tame}. In the spirit of \cite{ge2015escaping,lee2016}, we use the shorthand strict saddle to denote a saddle point where the Hessian has a negative eigenvalue (originally referred to as linearly unstable by Pemantle \cite[Theorem 1]{pemantle1990nonconvergence}). We next illustrate Corollary \ref{cor:local_min} using matrix factorization.

\begin{example}
\label{eg:pca}
Let $M \in \mathbb{R}^{m\times n}$ and consider the function
\begin{equation*}
    \begin{array}{cccc}
    f : & \mathbb{R}^{m\times r}\times \mathbb{R}^{n\times r} & \rightarrow & \mathbb{R}\\
        & (X,Y) & \mapsto & \|XY^\top-M\|^2
    \end{array}
\end{equation*}
where $m,n,r\in \mathbb{N}^*:=\{1,2,\hdots\}$. In this example, $\|\cdot\|=\sqrt{
\langle 
\cdot,\cdot\rangle}$ is the Frobenius norm and $\|\cdot\|_2$ is the spectral norm. The function is not coercive and its gradient is not Lipschitz continuous. Thus none of the convergence theorems in the optimization literature apply, to the best of our knowledge. There only exist results tailored to low-rank matrix factorization \cite[Theorem 3.1]{du2018algorithmic} \cite[Theorem 1.1]{ye2021global}. They require initialization near the origin and $\text{rank}(M) = r$, and only guarantee convergence with high probability. Fortunately, the function is semi-algebraic, hence definable in the real field with constants, and its saddle points are strict \cite{baldi1989neural,valavi2020revisiting} \cite[Theorem 3.3]{valavi2020landscape}. It also has bounded continuous gradient trajectories. Indeed, let $X:[0,T)\rightarrow\mathbb{R}^{m\times r}$ and $Y:[0,T)\rightarrow\mathbb{R}^{n\times r}$ where $0<T\leqslant \infty$ be differentiable functions such that
\begin{equation*}
    \dot{X} = - 2(XY^\top-M)Y,~~~ \dot{Y} = -2(XY^\top-M)^\top X.
\end{equation*}
On the one hand,
\begin{equation*}
    \dot{\overline{X^\top X-Y^\top Y}} = \dot{X}^\top X + X^\top \dot{X} - \dot{Y}^\top Y - Y^\top \dot{Y} = 0
\end{equation*}
and thus $X^\top X-Y^\top Y$ is constant (property known as balance \cite{du2018algorithmic,arora2018optimization,arora2019implicit}). On the other hand,
\begin{align*}
   \|X\|_2^4 + \|Y\|_2^4 ~ \leqslant & ~ \|X^\top X\|^2 + \|Y^\top Y\|^2 \\
    = & ~ \|X^\top X-Y^\top Y\|^2 + 2\langle X^\top X,Y^\top Y\rangle \\
    = & ~ \|X^\top X-Y^\top Y\|^2 + 2\|XY^\top \|^2 \\
    \leqslant & ~ \|X^\top X-Y^\top Y\|^2 + 2(\|XY^\top -M\|+\|M\|)^2 \\
    \leqslant & ~ \|X(0)^\top X(0)-Y(0)^\top Y(0)\|^2 + \\ 
    & ~ 2(\|X(0)Y(0)^\top -M\|+\|M\|)^2.
\end{align*}
Hence $X$ and $Y$ are bounded. By Corollary \ref{cor:local_min}, the gradient method with constant step size initialized in any bounded set converges almost surely to a local minimum. In fact, since $f$ has no spurious local minima, as was discovered in 1989 \cite{baldi1989neural}, we have established convergence to a global minimum almost surely.
\end{example}

Boundedness of continuous subgradient trajectories, which holds for lower semicontinuous convex functions that admit a minimum \cite[Theorem 4]{bruck1975asymptotic}, is a property shared by several other applications. Indeed, it can be shown that linear neural networks \cite[Proposition 1]{chitour2022geometric} \cite[Theorem 3.2]{bah2022learning} \cite[Proposition 4.1]{joszli2022} and some nonlinear neural networks with the sigmoid activation function have bounded continuous gradient trajectories \cite[Proposition 4.2]{joszli2022} (see \cite{du2019gradient} for recent work using the ReLU activation). Theorem \ref{thm:convergence} implies convergence to a critical point, which may not be a local minimum since not all saddle points are strict \cite{kawaguchi2016}. Matrix sensing also has bounded continuous gradient trajectories \cite[Proposition 4.4]{joszli2022} under the restricted isometry property (RIP) due to Candès and Tao \cite{candes2005decoding}. In addition, all critical points are either global minima or strict saddles \cite[Theorem III.1]{li2020global} provided that the RIP constant is less that 1/5. Corollary \ref{cor:local_min} again implies convergence to a global minimum almost surely.

In order to prove Theorem \ref{thm:convergence}, we propose the following result regarding real-valued multivariate functions.
\begin{lemma}
\label{lemma:uniform_length}
If a locally Lipschitz definable function has bounded continuous subgradient trajectories, then continuous subgradient trajectories initialized in a bounded set have uniformly bounded lengths.
\end{lemma}

Figuratively, Lemma \ref{lemma:uniform_length} says that lava flowing from the vent of a volcano always stays in a bounded region.  A key idea in the proofs of Lemma \ref{lemma:uniform_length} and Theorem \ref{thm:convergence} is to consider the supremum of the lengths of subgradient trajectories over all possible initial points in a bounded set, as well as over all possible step sizes in the discrete case (see \eqref{eq:sup_trajectory} and \eqref{eq:sup_discrete_X_0}). We then construct a sequence of bounded sets that iteratively exclude critical values of the objective function until none remain. In order to do so, we rely on the Kurdyka-\L{}ojasiewicz inequality \cite[Theorem 1]{kurdyka1998gradients}, the Morse-Sard theorem \cite{morse1939behavior,sard1942measure}, and the monotonicity theorem \cite[(1.2) p. 43]{van1998tame} \cite{pillay1986definable}.

We in fact propose a uniform Kurdyka-\L{}ojasiewicz inequality (see Proposition \ref{prop:KL}) better suited for our purposes. It should not to be confused with the uniformized KL property \cite[Lemma 6]{bolte2014proximal} which extends a pointwise version of the Kurdyka-\L{}ojasiewicz inequality to compact sets. The Kurdyka-\L{}ojasiewicz inequality guarantees the existence of a desingularizing function around each critical value of a locally Lipschitz definable function. By extending the inequality to hold across multiple critical values, we refine an important result of Kurdyka \cite[Theorem 2]{kurdyka1998gradients} 
regarding the length of gradient trajectories (see Proposition \ref{prop:length_continuous}). Our proof is also new because we integrate the Kurdyka-\L{}ojasiewicz inequality along a subgradient trajectory (see \eqref{eq:god_a}) and we use Jensen's inequality \cite{jensen1906fonctions}.

Kurdyka's aforementioned result states that continuous gradient trajectories of a continuously differentiable definable function that lie in a bounded set have uniformly bounded lengths, generalizing \L{}ojasiewicz' discovery for real analytic functions \cite{lojasiewicz1984trajectoires} \cite[Theor\`eme 5]{lojasiewicz1963propriete}. We make explicit the link between the length and the function variation using the desinguralizing function and the number of critical values, in both continuous (Proposition \ref{prop:length_continuous}) and discrete cases (Proposition \ref{prop:length_discrete}). 

This link, refined in the continuous case, new in the discrete case, enables us to determine how subgradient trajectories behave in the vicinity of saddle points. Either they stay above the corresponding critical value, or they admit a uniform decrease below it (see Example \ref{eg:decrease}). By the definable Morse-Sard theorem \cite[Corollary 9]{bolte2007clarke}, lower semicontinuous definable functions have finitely many critical values. In order to navigate from one to the next in the proofs of Lemma \ref{lemma:uniform_length} and Theorem \ref{thm:convergence}, we reason by induction in a way that is reminiscent of Bellman's principle of optimality \cite[p. 83]{bellman1957}.


This paper is organized as follows. Section \ref{sec:Literature review} contains a literature review. Section \ref{sec:Subgradient trajectories} contains some definitions and basic properties of subgradient trajectories. Section \ref{sec:Uniform Kurdyka-Lojasiewicz inequality} contains the uniform Kurdyka-\L{}ojasiewicz inequality and some consequences. Section \ref{sec:Proof of Lemma lemma:uniform_length} contains the proof of Lemma \ref{lemma:uniform_length}. We then prove Theorem \ref{thm:convergence} in a cyclic manner in Section \ref{sec:Proof of Theorem thm:convergence}.

\section{Literature review}
\label{sec:Literature review}
The gradient method was proposed by Cauchy in 1847 \cite{cauchy1847methode} for minimizing a nonnegative continuous multivariate function. In Lemaréchal's words \cite{lemarechal2012cauchy}, ``\textit{convergence is just sloppily mentioned}'' in \cite{cauchy1847methode} and ``\textit{we are now aware that some form of
uniformity is required from the objective’s continuity}''. According to Theorem \ref{thm:convergence}, no such assumption is required in the o-minimal setting. 

The first convergence proof is due to Temple \cite{temple1939general} in 1939. It deals with the case where the objective function is quadratic with a positive definite Hessian. He showed that the gradient converges to zero if the step size is chosen so as to minimize the function along the negative gradient (known as exact line search). Curry \cite{curry1944method} generalized this result to any function with continuous partial derivatives by taking the step size to correspond to the first stationary point along the negative gradient (known as Curry step). This stationary point is assumed to exist. It is also assumed that the iterates admit a converging subsequence. Interestingly, both assumptions hold if the objective is coercive. 

Subsequently, Kantorovich \cite{kantorovich1948functional} proposed his eponymous inequality which established a linear convergence rate for quadratic objective functions with positive definite Hessians. Akaike \cite{akaike1959successive} then argued that for ill-conditioned Hessians, the rate is generally close to its worst possible value, as observed by Forsythe and Motkzin \cite{forsythe1951asymptoticproperties}. Goldstein \cite{goldstein1965steepest}, Armijo \cite{armijo1966minimization}, and Wolfe \cite{wolfe1969convergence} later proposed inexact line search methods for differentiable functions that guarantee a linear decrease along the negative gradient. All their convergence theorems require some form of uniformity from the objective's continuity. These can be proven using Zoutendijk's condition \cite{zoutendijk1970nonlinear} and assuming Lipschitz continuity of the gradient on a sublevel set (see \cite[Theorem 3.2]{nocedal2006numerical} and \cite[Propositions 6.10-6.12]{gilbert2021fragments}). The latter assumption is also used in the convex setting \cite{polyak1987introduction,nesterov2018lectures}.

The assumption of Lipschitz continuity of the gradient seems to be omnipresent in the more general context of first-order methods. For the proximal gradient method, Bauschke et al. \cite{bauschke2017descent} recently proposed to exchange it with a Lipschitz-like/convexity condition when minimizing convex composite functions. This idea was generalized to nonconvex composite functions for a Bregman-based proximal gradient method \cite{bolte2018first}, where the new assumption is named smooth adaptable. The gradient method with constant step size is a special case of this algorithm, but in that case being smooth adaptable means having a Lipschitz continuous gradient. Note that the convergence theorem of the Bregman-based proximal gradient method \cite[Theorem 4.1]{bolte2018first} also requires the iterates to be bounded. A notion of stopping time \cite[Equation (3.1)]{patel2022gradient} was recently proposed in order to analyze the gradient method with diminishing step sizes without assuming Lipschitz continuity of the gradient. Several results are derived regarding the stopping time \cite[Theorems 3.3 and 4.1]{patel2022gradient}, but the results devoid of it again require the iterates to be bounded.

Once the iterates are assumed to be bounded, then a lot can be said about the gradient method. Absil et al. \cite[Theorem 4.1]{absil2005convergence} showed that bounded iterates of the gradient method with Wolfe's line search converge to a critical point if the objective function is analytic \cite[Theorem 4.1]{absil2005convergence}.  The proof relies on the \L{}ojasiewicz gradient inequality \cite[Proposition 1 p. 67]{law1965ensembles} and implies that the iterates have finite length, albeit not uniformly as in Theorem \ref{thm:convergence}. Note that without uniformity, Corollary \ref{cor:local_min} cannot be deduced. The exponent in the \L{}ojasiewicz gradient inequality also informs the convergence rate \cite{polyak1963gradient,li2018calculus}. A result of Attouch et al. \cite[Theorem 3.2]{attouch2013convergence} implies that bounded iterates of the gradient method with sufficiently small constant step sizes converge to a critical point if the objective function is differentiable with a Lipschitz continuous gradient and it satisfies the Kurdyka-\L{}ojasiewicz inequality at every point. If additionally the function is twice continuously differentiable and its Hessian has a negative eigenvalue at all saddle points and local maxima, then the limiting critical point must be a local minimum for almost every initial point, as shown by Lee et al. \cite[Proposition 14]{lee2016} and Panageas and Piliouras \cite[Theorem 3]{panageas2017} using the center and stable manifolds theorem. This convergence result was generalized to other first-order methods in \cite[Assumption 7]{lee2019}. Although not stated explicitly, the boundedness assumption is implicit in \cite[Theorem 3]{panageas2017} since the forward invariant domain needs to be bounded in order to guarantee convergence to a local minimum, as stated in the title of that paper.

In the main result of this paper, we establish an equivalence between boundedness of the iterates of the gradient method and boundedness of their continuous counterpart. The closest result in the literature seems to be \cite[Theorem 39]{bolte2010characterizations}. It considers a coercive differentiable convex objective function which satisfies the Kurdyka-\L{}ojasiewicz inequality and which has a Lipschitz continuous gradient. It shows that piecewise gradient iterates have uniformly bounded lengths if and only if piecewise gradient curves have uniformly bounded lengths, under the assumption that the iterates satisfy a strong descent assumption \cite[Equation (53)]{bolte2010characterizations} \cite[Definiton 3.1]{absil2005convergence}. The notions of piecewise gradient iterates and curves \cite[Definiton 15]{bolte2010characterizations} do not seem to inform the behavior of discrete and continuous gradient trajectories considered in this paper (see Definition \ref{def:trajectory}).

In order to establish global convergence of the gradient method, we propose Lemma \ref{lemma:uniform_length} as mentioned in the introduction. The length of subgradient trajectories has been of significant interest in the last decades. \L{}ojasiewicz \cite{lojasiewicz1984trajectoires} proved in the early eighties that bounded continuous gradient trajectories of analytic functions have finite length. This result has since been generalized to continuously differentiable definable functions by Kurdyka \cite[Theorem 2]{kurdyka1998gradients}, and later extended to nonsmooth settings by Bolte 
\textit{et al.} \cite[Theorem 3.1]{bolte2007lojasiewicz} \cite[Theorem 14]{bolte2007clarke}. Manselli and Pucci \cite[IX]{manselli1991maximum} \cite[Corollary 2.4]{daniilidis2015rectifiability} showed that continuous subgradient trajectories of multivariate convex functions that admit a minimum have finite length. This is false in infinite dimension due to a counterexample of Baillon \cite{baillon1978exemple}. These results relating to convexity solved an open problem posed by Br\'ezis \cite[Problem 13 p. 167]{brezis1973ope}. Discrete gradient trajectories of convex differentiable functions with Lipschitz continuous gradients also have finite length if the step size is constant and sufficiently small \cite[Corollary 3.9]{bohm2020ubiquitous} \cite[Theorem 15]{gupta2021path}.

It is conjectured that continuous subgradient trajectories of locally Lipschitz definable functions are nonoscillatory \cite[Conjecture F]{kurdyka1998gradients}, that is, the intersection of their orbit and any definable set has finitely many connected components. This would immediately imply Thom's gradient conjecture, which was shown to be true for analytic functions \cite{kurdyka2000proof}, and more generally for continuously differentiable functions definable in polynomially bounded o-minimal structures \cite{kurdyka43quasi}. In fact, it was shown that the radial projection of the difference between the trajectory and its limit has finite length.

Finally, let us discuss two relevant works \cite{d2021bounding,gupta2021path} published in 2021. The former bounds the lengths of continuous gradient trajectories that lie in a bounded set for families of definable functions and for polynomial functions, in which case the bound is explicit. The latter bounds the lengths of continuous and discrete gradient trajectories of differentiable functions with Lipschitz continuous gradients under the assumption of linear convergence to the set of global minimizers.

\section{Subgradient trajectories}
\label{sec:Subgradient trajectories}

We begin by stating some standard definitions. Let $B(a,r)$ and $S(a,r)$ respectively denote the closed ball and the sphere of center $a\in \mathbb{R}^n$ and radius $r \geqslant 0$.

\begin{definition}
\label{def:lipschitz}
A function $f:\mathbb{R}^n\rightarrow\mathbb{R}^m$ is locally Lipschitz if for all $a \in \mathbb{R}^n$, there exist $r>0$ and $L>0$ such that 
\begin{equation*}
        \forall x,y \in B(a,r), ~~~ \|f(x)-f(y)\| \leqslant L \|x-y\|.
    \end{equation*}
\end{definition}

\begin{definition}
    \label{def:Clarke}
    \cite[Chapter 2]{clarke1990}
    Let $f:\mathbb{R}^n \rightarrow \mathbb{R}$ be a locally Lipschitz function. The Clarke subdifferential is the set-valued mapping $\partial f:\mathbb{R}^n\rightrightarrows\mathbb{R}^n$ defined for all $x \in \mathbb{R}^n$ by
    $\partial f(x) := \{ s \in \mathbb{R}^n : f^\circ(x,h) \geqslant \langle s , h \rangle, ~ \forall h\in \mathbb{R}^n \}$ where
\begin{equation*}
    f^\circ(x,h) := \limsup_{\tiny\begin{array}{c} y\rightarrow x \\
    t \searrow 0
    \end{array}
    } \frac{f(y+th)-f(y)}{t}.
\end{equation*}
\end{definition}

We say that $x \in \mathbb{R}^n$ is critical if $0 \in \partial f(x)$. If $f$ is continuously differentiable in a neighborhood of $x \in \mathbb{R}^n$, then $\partial f(x) = \{\nabla f(x)\}$ \cite[2.2.4 Proposition p. 33]{clarke1990}.

\begin{definition}
\label{def:local_optimality}
A point $x^* \in \mathbb{R}^n$ is a local minimum (respectively global minimum) of a function $f:\mathbb{R}^n\rightarrow\mathbb{R}$ if there exists $\epsilon>0$ such that $f(x^*) \leqslant f(x)$ for all $x \in B(x^*,\epsilon)$ (respectively for all $x\in \mathbb{R}^n$). A local minimum $x^* \in \mathbb{R}^n$ is spurious if $f(x^*) > \inf\{ f(x) : x \in \mathbb{R}^n\}$. A point  $x^* \in \mathbb{R}^n$ is a saddle point if it is critical and it is neither a local minimum of $f$ nor $-f$.
\end{definition}

In order to define continuous subgradient trajectories, we recall the notion of absolute continuity.

\begin{definition}
    \label{def:absolute_continuity} 
    \cite[Definition 1 p. 12]{aubin1984differential}
    Given some real numbers $a\leqslant b$, a function $x:[a,b]\rightarrow\mathbb{R}^n$ is absolutely continuous if for all $\epsilon>0$, there exists $\delta>0$ such that, for any finite collection of disjoint subintervals $[a_1,b_1],\hdots,[a_m,b_m]$ of $[a,b]$ such that $\sum_{i=1}^m b_i-a_i \leqslant \delta$, we have $\sum_{i=1}^m \|x(b_i) - x(a_i)\| \leqslant \epsilon$.
\end{definition}

By virtue of \cite[Theorem 20.8]{nielsen1997introduction}, $x:[a,b]\rightarrow\mathbb{R}^n$ is absolutely continuous if and only if it is differentiable almost everywhere on $(a,b)$, its derivative $x'(\cdot)$ is Lebesgue integrable, and $x(t) - x(a) = \int_a^t x'(\tau)d\tau$ for all $t\in [a,b]$. Given a noncompact interval $I$ of $\mathbb{R}$, $x:I\rightarrow \mathbb{R}^n$ is absolutely continuous if it is absolutely continuous on all compact subintervals of $I$. Given an interval $I$ of $\mathbb{R}$, let $\mathcal{A}(I,\mathbb{R}^n)$ be the set of absolutely continuous functions defined from $I$ to $\mathbb{R}^n$.

\begin{definition}
\label{def:trajectory}
Let $f:\mathbb{R}^n \rightarrow \mathbb{R}$ be a locally Lipschitz function. Let $(S,\leqslant)$ be the set of absolutely continuous functions $x:[0,T)\rightarrow\mathbb{R}^n$ where $0<T\leqslant \infty$ such that  
\begin{equation*}
    x'(t) \in -\partial f(x(t)), ~~~\mathrm{for~a.e.}~t\in (0,T),
\end{equation*}
equipped with the partial order $\leqslant$ defined by $x_1:[0,T_1)\rightarrow\mathbb{R}^n$ is less than or equal to $x_2:[0,T_2)\rightarrow\mathbb{R}^n$ if and only if $T_1\leqslant T_2$ and $x_1(t) = x_2(t)$ for all $t\in [0,T_1)$. 
We refer to maximal elements of $(S,\leqslant)$ as continuous subgradient trajectories.
\end{definition}

We say that a continuous subgradient trajectory is globally defined if it is of the form $x:[0,T)\rightarrow\mathbb{R}^n$ where $T=\infty$. We refer to discrete subgradient trajectories as sequences $(x_k)_{k\in \mathbb{N}}$ in $\mathbb{R}^n$ such that $x_{k+1} \in x_k -\alpha_k \partial f(x_k)$ for all $k\in \mathbb{N}$ for some $\alpha_0,\alpha_1,\alpha_2,\hdots>0$. When $f$ is continuously differentiable, we refer to subgradient trajectories as  gradient trajectories.

\begin{definition}
\label{def:bounded_trajectory}
A continuous subgradient trajectory $x: [0,T) \rightarrow \mathbb{R}^n$ where $0< T\leqslant \infty$ is bounded if there exists $c>0$ such that $\|x(t)\| \leqslant c$ for all $t\in [0,T)$.
\end{definition}

In the rest of the section, we list basic properties of continuous subgradient trajectories that will be needed later. For other treatments of subgradient trajectories, see \cite[Chapter 17]{attouch2014variational} and \cite{garrigos2015descent}.

\begin{proposition}
\label{prop:existence_trajectory}
Locally Lipschitz functions have at least one continuous subgradient trajectory for every initial point.
\end{proposition}
\begin{proof}
Let $f:\mathbb{R}^n \rightarrow \mathbb{R}$ be a locally Lipschitz function and let $x_0\in\mathbb{R}^n$. Since $f$ is locally Lipschitz, the set-valued function $-\partial f$ is upper semicontinuous \cite[2.1.5 Proposition (d) p. 29]{clarke1990} with nonempty, compact and convex values \cite[2.1.2 Proposition (a) p. 27]{clarke1990}. By virtue of \cite[Theorem 3 p. 98]{aubin1984differential}, there exist $\epsilon>0$ and an absolutely continuous function $x:[0,\epsilon)\rightarrow\mathbb{R}^n$ such that  
\begin{equation*}
    x'(t) \in -\partial f(x(t)), ~\mathrm{for~a.e.}~ t\in (0,\epsilon),~~~ x(0) = x_0.
\end{equation*}
Let $S_{x_0}$ be the set of functions in $S$ from Definition \ref{def:trajectory} such that $x(0) = x_0$. Then $S_{x_0} \neq \emptyset$. Let $P$ be a nonempty totally ordered subset of $S_{x_0}$. Let $0<T \leqslant\infty$ be the supremum of the endpoints of the domains of all functions in $P$. Consider the function $x:[0,T)\rightarrow\mathbb{R}^n$ such that for all $\bar{x}:[0,\overline{T})\rightarrow \mathbb{R}^n$ in $P$ we have $x(t) := \bar{x}(t)$ for all $t\in [0,\overline{T})$. It is well defined and it is an upper bound of $P$ in $S_{x_0}$. By virtue of Zorn's lemma \cite[Corollary 2.5 p. 884]{lang2012algebra}, $S_{x_0}$ contains a maximal element. \end{proof}

\begin{proposition}
\label{prop:bounded_global}
Bounded continuous subgradient trajectories are globally defined.
\end{proposition}
\begin{proof}
Let $x:[0,T)\rightarrow\mathbb{R}^n$ be a bounded continuous subgradient trajectory of $f$ where $0<T \leqslant \infty$. We reason by contradiction and assume that $T<\infty$. By Definition \ref{def:bounded_trajectory}, there exists $c>0$ such that $\|x(t)\| \leqslant c$ for all $t\in [0,T)$. Consider a sequence $t_0,t_1,t_2,\hdots \in [0,T)$ converging to $T$. For all $k\in \mathbb{N}$, let $x_k:[0,T]\rightarrow\mathbb{R}^n$ be defined by $x_k(t) := x(t)$ for all $t\in [0,t_k]$ and $x_k(t) = 0$ for all $t\in (t_k,T]$. By construction, $x'_k$ converges pointwise almost everywhere to $x'$ on $(0,T)$. Let $L>0$ be a Lipschitz constant of $f$ on $B(0,c)$. By \cite[2.1.2 Proposition (a) p. 27]{clarke1990}, for all $x\in B(0,c)$ and $s\in \partial f(x)$, we have $\|s\| \leqslant L$. For almost every $t\in (0,T)$, we thus have $\|x_k'(t)\| \leqslant \|x'(t)\| \leqslant L$. By Lebesgue's dominated convergence theorem, $\int_0^{t_k} x'(t)dt = \int_0^T x'_k(t)dt \rightarrow \int_0^T x'(t)dt$. Since $x(\cdot)$ is absolutely continuous on $[0,t_k]$, we have 
$x(t_k) - x(0) = \int_0^{t_k} x'(t)dt$. Thus $x(\cdot)$ can be extended to an absolutely continuous function on $[0,T]$ where 
$x(T) := x(0) + \int_0^T x'(t)dt$. According to Proposition \ref{prop:existence_trajectory}, $x(\cdot)$ can further be extended to an absolutely continuous function on $[0,T+\epsilon)$ for some $\epsilon>0$ while satisfying the differential inclusion. This contradicts the maximality of $x(\cdot)$. 
\end{proof}

Proposition \ref{prop:equivalence_bounded} shows that the second statement in Theorem \ref{thm:convergence} is equivalent to requiring that all the continuous gradient trajectories are bounded, as stated in the abstract and Corollary \ref{cor:local_min}.

\begin{proposition}
\label{prop:equivalence_bounded}
Let $f:\mathbb{R}^n \rightarrow \mathbb{R}$ be a locally Lipschitz function and $x_0\in \mathbb{R}^n$. The following statements are equivalent:
\begin{enumerate}
    \item For all $0< T \leqslant \infty$ and any absolutely continuous function $x:[0,T)\rightarrow\mathbb{R}^n$ such that
    \begin{equation}
    \label{eq:[0,T)}
        x'(t) \in -\partial f(x(t)), ~~~ \mathrm{for~a.e.}~ t \in (0,T), ~~~ x(0) = x_0,
    \end{equation}
    there exists $c>0$ such that $\|x(t)\| \leqslant c$ for all $t\in [0,T)$.
    \item All the continuous subgradient trajectories of $f$ initialized at $x_0$ are bounded.
\end{enumerate}
\end{proposition}
\begin{proof}
($1\Longrightarrow 2$) Let $x:[0,T)\rightarrow\mathbb{R}^n$ be a continuous subgradient trajectory of $f$ where $x(0)=x_0$ and $0<T\leqslant \infty$. Since it is absolutely continuous and satisfies \eqref{eq:[0,T)}, there exists $c>0$ such that $\|x(t)\| \leqslant c$ for all $t\in [0,T)$. Thus the subgradient trajectory is bounded by Definition \ref{def:bounded_trajectory}. ($2\Longrightarrow 1$) Let $x:[0,T)\rightarrow\mathbb{R}^n$ be an absolutely continuous function where $x(0) = x_0$ and $0<T\leqslant \infty$ such that \eqref{eq:[0,T)} holds. Since the continuous subgradient trajectories initialized at $x_0$ are bounded, by Proposition \ref{prop:bounded_global} they are globally defined. Thus $x(\cdot)$ can be extended to an absolutely continuous function on $\mathbb{R}_+ := [0,\infty)$ while satisfying the differential inclusion. Again because the continuous subgradient trajectories are bounded, by Definition \ref{def:bounded_trajectory} there exists $c>0$ such that $\|x(t)\| \leqslant c$ for all $t\in \mathbb{R}_+$. 
\end{proof}

\begin{remark}
It may be tempting to only require $T=\infty$ in first statement of Proposition \ref{prop:equivalence_bounded} and the second statement of Theorem \ref{thm:convergence}, but a counterexample is given by $f(x)=-x^4/4$. Indeed, the only differentiable function $x:[0,T)\rightarrow\mathbb{R}$ such that \eqref{eq:[0,T)} holds with $T=\infty$ is the constant function equal to zero (when $x_0=0$), which is trivially bounded. The other gradient trajectories are neither bounded nor globally defined (they are of the form $x:[0,x_0^{-2}/2) \rightarrow \mathbb{R}$ with $x(t) := \mathrm{sign}(x_0)(x_0^{-2} -2t)^{-1/2}$ when $x_0\neq 0$).
\end{remark}

Proposition \ref{prop:tracking} guarantees that discrete trajectories track continuous trajectories up to a certain time in a uniform way with respect to the initial point. A key idea for proving Lemma \ref{lemma:uniform_length} and Theorem \ref{thm:convergence} alluded to in the introduction is also used to prove Proposition \ref{prop:tracking} (see \eqref{eq:sup_trajectory_T_smooth}).

\begin{proposition}
\label{prop:tracking}
Let $f:\mathbb{R}^n \rightarrow \mathbb{R}$ be a lower bounded differentiable function with a locally Lipschitz gradient. Let $X_0$ be a bounded subset of $\mathbb{R}^n$ and let $T>0$. For all $\epsilon>0$, there exists $\bar{\alpha}>0$ such that for all $\alpha_0,\alpha_1,\alpha_2,\hdots \in (0,\bar{\alpha}]$ and for all sequence $x_0,x_1,x_2,\hdots \in \mathbb{R}^n$ such that $x_0 \in X_0$ and $x_{k+1} = x_k - \alpha_k \nabla f(x_k)$ for all $k \in \mathbb{N}$, the continuous gradient trajectory $x:\mathbb{R}_+\rightarrow \mathbb{R}^n$ of $f$ initialized at $x_0$ satisfies
\begin{equation}
\label{eq:tracking}
   \forall k\in \mathbb{N}^*,~~~ \alpha_0+\cdots+\alpha_{k-1} \leqslant T ~~~\Longrightarrow~~~ \|x_k - x(\alpha_0+\cdots+\alpha_{k-1})\| \leqslant \epsilon.
\end{equation}
\end{proposition}

\begin{proof}
Without loss of generality, we may assume that $X_0 \neq \emptyset$. Let $L>0$ and $M>0$ respectively denote Lipschitz constants of $f$ and $\nabla f$ with respect to $\|\cdot\|$ on the convex hull of $B(X_0,\sigma_T(X_0)+1):= X_0 + B(0,\sigma_T(X_0)+1)$ where 
\begin{subequations}
\label{eq:sup_trajectory_T_smooth}
\begin{align}
    \sigma_T(X_0) := & \sup\limits_{x \in \mathcal{C}^1(\mathbb{R}_+,\mathbb{R}^n)} ~~ \int_0^T \|x'(t)\|dt \\
  & ~~ \mathrm{subject~to} ~~~ 
\left\{ 
\begin{array}{l}
x'(t) = - \nabla f(x(t)), ~\forall t > 0,\\[3mm] x(0) \in X_0.
\end{array}
\right.
\end{align}
\end{subequations}
Above, $\mathcal{C}^1(\mathbb{R}_+,\mathbb{R}^n)$ is the set of continuous functions on $[0,\infty )$ that are continuously differentiable on $(0,\infty)$. Note that $\sigma_T(X_0)>-\infty$ since $f$ is lower bounded and differentiable with a locally Lipschitz gradient. For every $x_0\in X_0$, there hence exists a unique globally defined continuous gradient trajectory of $f$ initialized at $x_0$ \cite[Theorem 17.1.1]{attouch2014variational}. Also, $\sigma_T(X_0)<\infty$ since for any feasible point $x(\cdot)$ of \eqref{eq:sup_trajectory_T_smooth}, we have 
\begin{subequations}
\begin{align}
    \int_0^T \|x'(t)\|dt & \leqslant \sqrt{T} \sqrt{ \int_0^T \|x'(t)\|^2 dt} \label{eq:length_finite_smooth_a} \\
    & = \sqrt{T} \sqrt{ -\int_0^T \langle \nabla f(x(t)), x'(t) \rangle dt} \label{eq:length_finite_smooth_b} \\
    & = \sqrt{T} \sqrt{ -\int_0^T (f \circ x)'(t) dt} \label{eq:length_finite_smooth_c} \\
    & = \sqrt{T} \sqrt{ f(x(0)) - f(x(T))} \label{eq:length_finite_smooth_d} \\
    & \leqslant \sqrt{T\left(\sup_{X_0} f - \inf_{\mathbb{R}^n} f\right)} < \infty. \label{eq:length_finite_smooth_e}
\end{align}
\end{subequations}
Indeed, \eqref{eq:length_finite_smooth_a} is due to the Cauchy-Schwarz inequality. \eqref{eq:length_finite_smooth_e} holds because $X_0$ is bounded and $f$ is lower bounded.

It is easy to check that $L$ and $ML$ are respectively Lipschitz and gradient Lipschitz constants on $[0,T]$ of all continuous gradient trajectories of $f$ initialized in $X_0$. Indeed, let $x:\mathbb{R}_+\rightarrow \mathbb{R}^n$ be a continuous gradient trajectory initialized in $X_0$. Since $x(t) \in B(X_0,\sigma_T(X_0))$ for all $t\in [0,T]$, we have $\|x'(t)\| = \|\nabla f(x(t))\| \leqslant L$. By the mean value theorem \cite[5.19 Theorem]{rudin1964principles}, for all $s,t \in [0,T]$ we have $\|x'(t)-x'(s)\| = \|\nabla f(x(t)) - \nabla f(x(s))\| \leqslant M \|x(t)-x(s)\| \leqslant M L |t-s|$. As a byproduct, we get the Taylor bound
\begin{subequations}
    \begin{align}
    \|x(t)-x(s) - x'(s)(t-s)\| & = \left\| \int_s^t (x'(\tau) - x'(s)) d\tau \right\| \\
    & \leqslant \left| \int_s^t \| x'(\tau) - x'(s) \| d\tau \right| \\
    & \leqslant \left|\int_s^t ML | \tau - s | d\tau \right| \\
    & = \frac{ML}{2}(t-s)^2. \label{eq:local_error}
    \end{align}
\end{subequations}

Let $\epsilon \in (0,1)$ and let $\bar{\alpha} := 2\epsilon e^{-MT}/(LMT)$. Consider $\alpha_0,\alpha_1,\alpha_2,\hdots \in (0,\bar{\alpha}]$ and a sequence $x_0,x_1,x_2,\hdots \in \mathbb{R}^n$ such that $x_{k+1} = x_k - \alpha_k \nabla f(x_k)$ for all $k \in \mathbb{N}$ and $x_0 \in X_0$. Let $x:\mathbb{R}_+\rightarrow \mathbb{R}^n$ be the unique continuous gradient trajectory of $f$ initialized at $x_0$. Let $t_0:=0$ and $t_k := \alpha_0+\cdots+\alpha_{k-1}$ when $k\in \mathbb{N}^*$ so that $\|x_0 - x(t_0)\| = 0 \leqslant \epsilon$. Assume that \eqref{eq:tracking} holds up to some index $K$. For $k=0,\hdots,K$, we have
\begin{subequations}
    \begin{align}
        \|x_{k+1} - x(t_{k+1})\| = &  ~\| x_k - \alpha_k \nabla f(x_k) - x(t_{k+1}) \| \label{local_a} \\[2mm]
        \leqslant  & ~ \| x_k - \alpha_k \nabla f(x_k) - [x(t_k) - \alpha_k \nabla f(x(t_k))] \|~+ \label{local_b} \\[2mm] & ~ \| x(t_k) - \alpha_k \nabla f(x(t_k)) - x(t_{k+1}) \| \label{local_c} \\[2mm]
        \leqslant & ~ \|x_k - x(t_k)\| + \alpha_k \|\nabla f(x_k) - \nabla f(x(t_k))\| ~+ \label{local_d} \\[2mm]
        & ~ ML\alpha_k^2/2 \label{local_e} \\[2mm]
        \leqslant & ~ (1+\alpha_k M)\|x_k - x(t_k)\| + ML\alpha_k^2/2. \label{local_f}
    \end{align}
\end{subequations}
The term in \eqref{local_c} is equal to the local truncation error $\|x(t_{k+1})  - x(t_k) - x'(t_k)(t_{k+1}-t_k) \|$, and can hence be bounded above using \eqref{eq:local_error}. In the second term of \eqref{local_d}, we invoke a Lipschitz constant of $\nabla f$ on the convex hull of $B(X_0,\sigma_T(X_0)+1)$. Indeed, by the induction hypothesis and $\epsilon < 1$, $x_k$ and $x(t_k)$ belong to $B(X_0,\sigma_T(X_0)+\epsilon)$. Hence 
\begin{subequations}
    \begin{align}
        \|x_{K+1} - x(t_{K+1})\| \leqslant &~ \prod_{k=0}^K (1+\alpha_kM)\|x_0-x(t_0)\|~+ \label{eq:global_a}\\
        &~ \sum_{k=0}^K \frac{ML\alpha_k^2}{2} \prod_{l=k+1}^{K} (1+\alpha_lM)\label{eq:global_b}\\
        \leqslant &~ \prod_{l=0}^{K} (1+\alpha_lM) \sum_{k=0}^K \frac{ML\alpha_k^2}{2}\label{eq:global_c}\\
        \leqslant &~ \left(1+\frac{M}{K+1}\sum_{k=0}^K \alpha_k\right)^{K+1} \sum_{k=0}^K \frac{ML\bar{\alpha}\alpha_k}{2}\label{eq:global_d}\\
        \leqslant &~ \left(1+\frac{MT}{K+1}\right)^{K+1} \frac{ML\bar{\alpha}T}{2}\label{eq:global_e}\\
        \leqslant &~~ \bar{\alpha} LMTe^{MT}/2\label{eq:global_f}\\
        = &~~ \epsilon.\label{eq:global_g}
    \end{align}
\end{subequations}
The term in \eqref{eq:global_a} is equal to zero because $x_0 = x(t_0)$. By convention, we set the product in \eqref{eq:global_b} to be equal to one if the index set for $l$ is empty. We bound the partial products ranging from $k+1$ to $K$ inside the summation of \eqref{eq:global_b} by the full product ranging from $0$ to $K$ in order to obtain \eqref{eq:global_c}. We then use the inequality of arithmetic and geometric means and $\alpha_k\leqslant\bar{\alpha}$ in order to obtain \eqref{eq:global_d}. In order to prove the induction step, we assume that $\alpha_0+\cdots+\alpha_K \leqslant T$, hence \eqref{eq:global_e}. \eqref{eq:global_f} follows from $\ln(1+x) \leqslant x$ for all $x>0$. Finally, \eqref{eq:global_g} holds because $\bar{\alpha} = 2\epsilon e^{-MT}/(LMT)$.
\end{proof}

\begin{remark}
\label{rem:tracking}
If $X_0$ is a singleton $\{x_0\}$ and $x:[0,\hat{T})\rightarrow \mathbb{R}^n$ is a gradient trajectory initialized at $x_0$, then one may remove the assumption that $f$ is lower bounded in Proposition \ref{prop:tracking} and instead require that $T \in (0,\hat{T})$. In that case, $\sigma_T(X_0)=\int_0^T \|x'(t)\|dt < \infty$ since $\nabla f$ is Lipschitz continuous on the bounded set $x([0,T])$.
\end{remark}

\section{Uniform Kurdyka-\L{}ojasiewicz inequality}
\label{sec:Uniform Kurdyka-Lojasiewicz inequality}

The Kurdyka-\L{}ojasiewicz inequality generalizes the \L{}ojasiewicz gradient inequality \cite[Proposition 1 p. 67]{law1965ensembles} from lower semicontinuous functions definable in polynomially bounded o-minimal structures \cite{loi2016lojasiewicz} 
to lower semicontinuous functions definable in arbitrary o-minimal structures \cite[Theorem 14]{bolte2007clarke}. 

O-minimal structures (short for order-minimal) were originally considered in \cite{van1984remarks,pillay1986definable}. They are founded on the observation that many properties of semi-algebraic sets can be deduced from a few simple axioms \cite{van1998tame}. Recall that a subset $A$ of $\mathbb{R}^n$ is semi-algebraic \cite{bochnak2013real} if it is a finite union of basic semi-algebraic sets, which are of the form $\{ x \in \mathbb{R}^n : p_i(x) = 0, ~ i = 1,\hdots,k; ~  p_i(x) > 0, ~ i = k+1,\hdots,m \}$ where $p_1,\hdots,p_m \in \mathbb{R}[X_1,\hdots,X_n]$ (i.e., polynomials with real coefficients). 

\begin{definition}
\label{def:o-minimal}
\cite[Definition p. 503-506]{van1996geometric}
An o-minimal structure on the real field is a sequence $S = (S_k)_{k \in \mathbb{N}}$ such that for all $k \in \mathbb{N}$:
\begin{enumerate}
    \item $S_k$ is a boolean algebra of subsets of $\mathbb{R}^k$, with $\mathbb{R}^k \in S_k$;
    \item $S_k$ contains the diagonal $\{(x_1,\hdots,x_k) \in \mathbb{R}^k : x_i = x_j\}$ for $1\leqslant i<j \leqslant k$;
    \item If $A\in S_k$, then $A\times \mathbb{R}$ and $\mathbb{R}\times A$ belong to $S_{k+1}$;
    \item If $A \in S_{k+1}$ and $\pi:\mathbb{R}^{k+1}\rightarrow\mathbb{R}^k$ is the projection onto the first $k$ coordinates, then $\pi(A) \in S_k$;
    \item $S_3$ contains the graphs of addition and multiplication;
    \item $S_1$ consists exactly of the finite unions of open intervals and singletons. 
\end{enumerate}
\end{definition}

A subset $A$ of $\mathbb{R}^n$ is definable in an o-minimal structure $(S_k)_{k\in\mathbb{N}}$ if $A \in S_k$. A function $f:\mathbb{R}^n\rightarrow\mathbb{R}$ is definable in an o-minimal structure if its graph, that is to say $\{ (x,t) \in \mathbb{R}^{n+1} : f(x)=t \}$, is definable in that structure.
Examples of o-minimal structures include the real field with constants, whose definable sets are the semi-algebraic sets (by Tarski-Seidenberg \cite{tarski1951decision,seidenberg1954new}), the real field with restricted analytic functions, whose definable sets are the globally subanalytic sets (by Gabrielov \cite{gabrielov1968projections,van1986generalization}), the real field with the exponential function (by Wilkie \cite{wilkie1996model}), the real field with the exponential and restricted analytic functions (by van den Dries, Macintyre, and Marker \cite{van1994elementary}), the real field with restricted analytic and real power functions (by Miller \cite{miller1994expansions}), and the real field with convergent generalized power series (by van den Dries and Speissegger \cite{van1998real}). Note that there is no largest o-minimal structure \cite{rolin2003quasianalytic}. Recall that throughout this paper, we fix an arbitrary o-minimal structure $(S_k)_{k\in\mathbb{N}}$.

We next state the uniform Kurdyka-\L{}ojasiewicz inequality. Given a subset of $S$ of $\mathbb{R}^n$, let $\mathring{S}$ and $\overline{S}$ denote the interior and closure of $S$ in $\mathbb{R}^n$ respectively. A function $\psi:S \rightarrow S$ is a homeomorphism if it is a continuous bijection and the inverse function $\psi^{-1}$ is continuous. $\psi:S \rightarrow S$ is a diffeomorphism if $\mathring{S} \neq \emptyset$, $\psi$ is a homeomorphism, and both $\psi$ and $\psi^{-1}$ are continuously differentiable on $\mathring{S}$. Given $x\in \mathbb{R}^n$, consider the distance of $x$ to $S$ defined by $d(x,S) := \inf \{ \|x-y\| : y \in S \}$. Given a locally Lipschitz function $f:\mathbb{R}^n\rightarrow\mathbb{R}$, a real number $v$ is critical value of $f$ in $S$ if there exists $x \in S$ such that $v = f(x)$ and $0\in \partial f(x)$. Given a set-valued mapping $F:\mathbb{R}^n\rightrightarrows\mathbb{R}^m$ and $y \in \mathbb{R}^m$, let $F^{-1}(y) := \{ x \in \mathbb{R}^n : F(x) \ni y \}$.

\begin{proposition}
\label{prop:KL}
Let $f:\mathbb{R}^n \rightarrow \mathbb{R}$ be a locally Lipschitz definable function and let $X$ be a bounded subset of $\mathbb{R}^n$. Let $V$ be the set of critical values of $f$ in $\overline{X}$ if it is nonempty, otherwise let $V:=\{0\}$. There exists a concave definable diffeomorphism $\psi:[0,\infty)\rightarrow[0,\infty)$ such that
\begin{equation*}
    \forall x \in X \setminus (\partial \tilde{f})^{-1}(0), ~~~ d(0,\partial (\psi \circ \tilde{f})(x)) \geqslant 1
\end{equation*}
where $\tilde{f}(x) := d(f(x),V)$ for all $x\in \mathbb{R}^n$.
\end{proposition}
\begin{proof}
By the definable Morse-Sard theorem \cite[Corollary 9]{bolte2007clarke}, the set $V$ is finite. In addition, by the Kurdyka-\L{}ojasiewicz inequality \cite[Corollary 15]{bolte2007clarke}, there exist $\epsilon>0$ and a strictly increasing concave continuous definable function $\psi:[0,\epsilon)\rightarrow \mathbb{R}$ that is continuously differentiable on $(0,\epsilon)$ with $\psi(0) = 0$ such that, for all $x \in X$ and $v \in V$ satisfying $0 < | f(x)-v | < \epsilon$, we have $d(0,\partial (\psi \circ | f-v|)(x)) \geqslant 1$. The fact that $\psi$ can be assumed to be concave is due to the monotonicity theorem \cite[(1.2) p. 43]{van1998tame} \cite[Lemma 2]{kurdyka1998gradients} by following the argument in \cite[Theorem 4.1]{attouch2010proximal}. Let $\rho:= \inf \{(v-v')/2: v,v' \in V, v > v'\}>0$ and $\rho' : = \min \{ \rho, \epsilon\}$. Observe that $c>0$ where
\begin{equation*}
    c := \inf \{ \|s\| : 
s \in \partial f(x), ~ x\in X, ~ | f(x) - v | \geqslant \rho',~ v\in V \}.
\end{equation*}
Indeed, assume that there exist sequences $(s_k)_{k\in \mathbb{N}}$ and $(x_k)_{k\in\mathbb{N}}$ such that $\|s_k\|$ converges to zero and $s_k \in \partial f(x_k)$, $x_k\in X$, and $| f(x_k) - v | \geqslant \rho'$ for all $v\in V$. Since $X$ is bounded, there exist subsequences (again denoted $s_k$ and $x_k$) such that $x_k$ has a limit $x^*$ in $\overline{X}$. Since $f$ is continuous, $| f(x^*) - v | \geqslant \rho'>0$ for all $v \in V$. Also, since $s_k$ converges to zero, by \cite[2.1.5 Proposition (b) p. 29]{clarke1990}, we have $0\in \partial f(x^*)$. This yields a contradiction. 

After possibly reducing $\epsilon$, we may assume that $\lim_{t\nearrow\rho'} \psi'(t)>0$. If $\lim_{t\nearrow\rho'} \psi'(t) \geqslant 1/c$, then we may extend $\psi$ to an affine function on $[\rho',\infty)$ with slope equal to $\lim_{t\nearrow\rho'} \psi'(t)$. Otherwise, we can multiply $\psi$ by $1/(c\lim_{t\nearrow\rho'} \psi'(t))$ before taking the affine extension. Now let $x \in X$ be such that $0 \notin \partial \tilde{f}(x)$. If $f(x)$ is strictly within $\rho'$ distance of some $v \in V$, then from $\rho' \leqslant \rho$ it follows that $\tilde{f}(\tilde{x}) = d(f(\tilde{x}),V) = | f(\tilde{x})-v|$ for all $\tilde{x}$ in a neighborhood of $x$ in $\mathbb{R}^n$, and thus $d(0,\partial ( \psi \circ \tilde{f})(x)) = d(0,\partial (\psi \circ | f-v|)(x)) \geqslant 1$. The inequality is due to $\rho' \leqslant \epsilon$. Otherwise, by \cite[2.3.9 Theorem (Chain Rule I) (ii) p. 42]{clarke1990} we have $d(0,\partial ( \psi \circ \tilde{f})(x)) = \psi'(\tilde{f}(x)) d(0,\partial \tilde{f} (x)) = \psi'(\tilde{f}(x)) d(0,\partial f(x)) \geqslant \psi'(\tilde{f}(x)) c \geqslant 1$. 
\end{proof}

We say that $\psi$ in Proposition \ref{prop:KL} is a desingularizing function of $f$ on $X$. In order to prove Proposition \ref{prop:length_continuous}, we recall the following result.

\begin{proposition}
\label{prop:chain_rule}
Let $f:\mathbb{R}^n \rightarrow \mathbb{R}$ be a definable function and let $x:(0,T)\rightarrow \mathbb{R}^n$ with $0 < T \leqslant \infty$ be an absolutely continuous function. If $f$ is locally Lipschitz at $x(t)$ for all $t \in (0,T)$ and $x'(t) \in -\partial f(x(t))$ for almost every $t\in(0,T)$, then
\begin{equation*}
    (f \circ x)'(t) = -\|x'(t)\|^2 = -d(0,\partial f(x(t)))^2,~~~\mathrm{for~a.e.}~t \in (0,T).
\end{equation*}
\end{proposition}
\begin{proof}
Since $f$ is definable and locally Lipschitz at $x(t)$ for all $t\in(0,T)$, it satisfies the chain rule \cite[Proposition 2 (iv)]{bolte2020conservative} (see also \cite[Theorem 5.8]{davis2020stochastic} and \cite[Lemma 2.10]{drusvyatskiy2015quadratic}), that is to say $(f \circ x)'(t) = \langle s , x'(t) \rangle$ for all $s\in \partial f(x(t))$ and almost every $t\in (0,T)$. In particular, we may take $s := - x'(t) \in \partial f(x(t))$.
Finally, by \cite[Lemma 5.2]{davis2020stochastic} (see also \cite[Proposition 4.10]{drusvyatskiy2015curves}) it holds that $\|x'(t)\| = d(0,\partial f(x(t)))$ for almost every $t\in (0,T)$. 
\end{proof}

The following proposition is a refinement of \cite[Theorem 2]{kurdyka1998gradients}, as discussed in the introduction.

\begin{proposition}
\label{prop:length_continuous}
Let $f:\mathbb{R}^n \rightarrow \mathbb{R}$ be a locally Lipschitz definable function and let $X$ be a bounded subset of $\mathbb{R}^n$. Assume that $f$ has at most $m\in \mathbb{N}^*$ critical values in $\overline{X}$. Let $\psi$ be a desingularizing function of $f$ on $X$. If $x:[0,T]\rightarrow X$ with $T \geqslant 0$ is an absolutely continuous function such that $x'(t) \in -\partial f(x(t))$ for almost every $t\in(0,T)$, then
\begin{equation}
\label{eq:length_continuous}
    \frac{1}{2m}\int_0^T\|x'(t)\|dt ~\leqslant~ \psi\left(\frac{f(x(0))-f(x(T))}{2m}\right).
\end{equation}
\end{proposition}
\begin{proof}
We first consider the special case where $0 \notin \partial \tilde{f}(x(t))$ for all $t\in (0,T)$ where $\tilde{f}$ is defined as in Proposition \ref{prop:KL}. Consider the change of variables $\bar{x} = x \circ \varphi$ where $\bar{t}=\varphi^{-1}(t) = \int_0^{t} ds /(\psi' \circ \tilde{f} \circ x)(s)$ and $t=\varphi(\bar{t}) = \int_0^{\bar{t}} (\psi' \circ \tilde{f} \circ \bar{x})(\bar{s}) d\bar{s}$ by the inverse function theorem \cite[Theorem 17.7.2]{tao2006analysis}. We have $\bar{x}'(\bar{t}) = \varphi'(\bar{t})x'(t) \in -(\psi' \circ \tilde{f} \circ \bar{x})(\bar{t}) \partial f(\bar{x}(\bar{t})) = \pm \partial (\psi \circ \tilde{f})(\bar{x}(\bar{t}))$ for almost every $\bar{t}\in (0,\overline{T})$, where $\overline{T} = \varphi^{-1}(T)$ and the sign is constant. By Proposition \ref{prop:chain_rule} and the Cauchy-Schwarz inequality, we thus have
\begin{subequations}
\begin{align}
    \overline{T} \leqslant & \int_0^{\overline{T}} d(0,\partial (\psi \circ \tilde{f})(\bar{x}(\bar{t})))d\bar{t} \label{eq:god_a} \\
    = & \int_0^{\overline{T}} \|\bar{x}'(\bar{t})\|d\bar{t} \label{eq:god_b} \\
    = & \int_0^{T} \|x'(t)\|dt \\
    \leqslant & \sqrt{\overline{T}} \sqrt{\int_0^{\overline{T}} \|\bar{x}'(\bar{t})\|^2d\bar{t}} \\
    = & \sqrt{\overline{T}}\sqrt{|(\psi\circ \tilde{f}\circ\bar{x})(0)-(\psi\circ \tilde{f}\circ\bar{x})(\overline{T})|}.
\end{align}
\end{subequations}
Hence $\overline{T} \leqslant |(\psi\circ \tilde{f}\circ\bar{x})(0)-(\psi\circ \tilde{f}\circ\bar{x})(\overline{T})|$ and, since $\psi$ is concave, we have
\begin{align*}
     \int_0^T \|x'(t)\|dt & \leqslant |(\psi\circ \tilde{f}\circ\bar{x})(0)-(\psi\circ \tilde{f}\circ\bar{x})(\overline{T})| \\
     & \leqslant \psi (|( \tilde{f}\circ\bar{x})(0)-(\tilde{f}\circ\bar{x})(\overline{T})|) \\[2mm]
     & = \psi (f(x(0))-f(x(T))).
\end{align*}
We next consider the general case where there may exist $t\in (0,T)$ such that $0 \in \partial \tilde{f}(x(t))$, namely $0 < t_1 < \cdots < t_{k} < T$, and potentially $(t_k,T)$. For notational convenience, let $t_0 := 0$ and $t_{k+1}:=T$ where $k$ is possibly equal to zero. Since $k \leqslant 2m-1$ and $\psi$ is concave, we have
\begin{align*}
    \int_0^T \|x'(t)\|dt & = \sum_{i=0}^k \int_{t_i}^{t_{i+1}} \|x'(t)\|dt \\
    & \leqslant \sum_{i=0}^k \psi (f(x(t_i))-f(x(t_{i+1}))) \\[-.5mm]
    & \leqslant (k+1) ~ \psi \left( \frac{1}{k+1} \sum_{i=0}^k f(x(t_i))-f(x(t_{i+1}))\right) \\
    & = (k+1) ~ \psi \left( \frac{f(x(0)) - f(x(T))}{k+1} \right) \\
    & \leqslant 2m ~ \psi \left( \frac{f(x(0)) - f(x(T))}{2m} \right). \qedhere
\end{align*}
\end{proof}

We next propose a discrete version of Proposition \ref{prop:length_continuous}.

\begin{proposition}
\label{prop:length_discrete}
Let $f:\mathbb{R}^n \rightarrow \mathbb{R}$ be a definable differentiable function with a locally Lipschitz gradient and let $X$ be a bounded open subset of $\mathbb{R}^n$. Assume that $f$ has at most $m\in \mathbb{N}^*$ critical values in $\overline{X}$ and that $L>0$ is a Lipschitz constant of $f$ on the convex hull of $X$. Let $\psi$ be a desingularizing function of $f$ on $X$ and $\epsilon>0$. There exists $\bar{\alpha}>0$ such that, for all $K \in \mathbb{N}$, $\alpha_0,\hdots,\alpha_K \in (0,\bar{\alpha}]$, and $(x_0,\hdots,x_{K+1}) \in X \times \cdots \times X \times \mathbb{R}^n$ such that $x_{k+1} = x_k - \alpha_k \nabla f(x_k)$ for $k = 0,\hdots,K$, we have
\begin{equation}
\label{eq:length_discrete}
    \frac{1}{(2+\epsilon)m}\sum_{k=0}^K\|x_{k+1}-x_k\| ~\leqslant~ \psi\left(\frac{f(x_0)-f(x_K)}{2m}\right) + \frac{2L}{2+\epsilon}\max_{k=0,\hdots,K} \alpha_k
\end{equation}
and $f(x_0)\geqslant \cdots \geqslant f(x_{K+1})$.
\end{proposition}
\begin{proof}
Let $M>0$ be a Lipschitz constant of $\nabla f$ with respect to $\|\cdot\|$ on the convex hull of $B(X,1):=X+B(0,1)$ and let $\bar{\alpha} := \min\{L^{-1},2\epsilon(6+\epsilon)^{-1}M^{-1}\}$. Let $K \in \mathbb{N}$, $\alpha_0,\hdots,\alpha_K \in (0,\bar{\alpha}]$, and $(x_0,\hdots,x_{K+1}) \in X \times \cdots \times X \times \mathbb{R}^n$ be such that $x_{k+1} = x_k - \alpha_k \nabla f(x_k)$ for $k=0,\hdots,K$. It holds that $\|x_{K+1}-x_K\| = \alpha_K\|\nabla f(x_K)\| \leqslant \bar{\alpha}L \leqslant 1$ and thus $x_0,\hdots,x_{K+1} \in B(X,1)$. A bound on the Taylor expansion of $f$ yields
\begin{align*}
    f(x_{k+1})-f(x_k) & \leqslant \langle \nabla f(x_k), x_{k+1}-x_k\rangle + \frac{M}{2}\|x_{k+1}-x_k\|^2 \\
    & = \left(\frac{M\alpha_k}{2}-1\right) \|x_{k+1}-x_k\| \|\nabla f(x_k)\|
\end{align*}
for $k=0,\hdots,K$, and thus
\begin{equation}
\label{eq:descent_lemma_1}
    \|x_{k+1}-x_k\|\|\nabla f(x_k)\| \leqslant \frac{2}{2- M\alpha_k}(f(x_k) - f(x_{k+1})).
\end{equation}
We also have
    \begin{align*}
        \|\nabla f(x_{k+1})\| & \leqslant \| \nabla f(x_{k+1}) - \nabla f(x_k) \| + \|\nabla f(x_k)\| \\
        & \leqslant M \| x_{k+1} - x_k \| + \|\nabla f(x_k)\| \\
        & \leqslant (M\alpha_k+1) \|\nabla f(x_k)\|
    \end{align*}
and thus 
\begin{equation}
\label{eq:descent_lemma_2}
    \|x_{k+1}-x_k\|\|\nabla f(x_{k+1})\| \leqslant \frac{2+2M\alpha_k}{2- M\alpha_k}(f(x_k) - f(x_{k+1})).
\end{equation}
Note that 
\begin{equation*}
    1 \leqslant \frac{2}{2- M\alpha_k} \leqslant \frac{2+2M\alpha_k}{2- M\alpha_k} \leqslant 1+\frac{\epsilon}{2}.
\end{equation*}
Let $\tilde{f}$ and $V$ be defined as in Proposition \ref{prop:KL}.

Assume that $[f(x_K),f(x_0)]$ excludes the elements of $V$ and the averages of any two consecutive elements of $V$. Then $0\notin \partial \tilde{f}(x_k)$ and $1 \leqslant \| \nabla(\psi \circ \tilde{f})(x_k)\| = \psi'(\tilde{f}(x_k)) \|\nabla \tilde{f}(x_k)\|$ for $k=0,\hdots,K$. Let $k \in \{0,\hdots,K-1\}$. If $\tilde{f}(x_k) \geqslant \tilde{f}(x_{k+1})$, then multiplying \eqref{eq:descent_lemma_1} by $\psi'(\tilde{f}(x_k))$ and using the concavity of $\psi$, we find that 
\begin{align*}
    \|x_{k+1}-x_k\| & \leqslant (1+\epsilon/2)\psi'(\tilde{f}(x_k))(f(x_k) - f(x_{k+1})) \\
    & = (1+\epsilon/2)\psi'(\tilde{f}(x_k))(\tilde{f}(x_k) - \tilde{f}(x_{k+1})) \\
    & \leqslant (1+\epsilon/2)(\psi (\tilde{f}(x_k)) - \psi (\tilde{f}(x_{k+1}))).
\end{align*}
If $\tilde{f}(x_k) \leqslant \tilde{f}(x_{k+1})$, then multiplying \eqref{eq:descent_lemma_2} by $\psi'(\tilde{f}(x_{k+1}))$ and using the concavity of $\psi$, we find that 
\begin{align*}
    \|x_{k+1}-x_k\| & \leqslant (1+\epsilon/2)\psi'(\tilde{f}(x_{k+1}))(f(x_k) - f(x_{k+1})) \\
    & = (1+\epsilon/2)\psi'(\tilde{f}(x_{k+1}))(\tilde{f}(x_{k+1}) - \tilde{f}(x_k)) \\
    & \leqslant (1+\epsilon/2)(\psi (\tilde{f}(x_{k+1})) - \psi (\tilde{f}(x_k))).
\end{align*}
As a result,
\begin{equation*}
    \|x_{k+1}-x_k\| \leqslant (1+\epsilon/2)|\psi (\tilde{f}(x_k)) - \psi (\tilde{f}(x_{k+1}))|, ~~~ k = 0,\hdots,K-1.
\end{equation*}
We obtain the telescoping sum 
\begin{subequations}
    \label{eq:length_finite} 
    \begin{align}
    \sum_{k=0}^K \|x_{k+1}-x_k\| & \leqslant (1+\epsilon/2) |\psi(\tilde{f}(x_0)) - \psi(\tilde{f}(x_K))| + \|x_{K+1} - x_K\| \\[-4mm]
    & \leqslant (1+\epsilon/2) \psi ( | \tilde{f}(x_0) - \tilde{f}(x_K) |) + \alpha_K\|\nabla f(x_K)\| \\
    & \leqslant (1+\epsilon/2) \psi( f(x_0) - f(x_K)) + L\max_{k=0,\hdots,K} \alpha_k.
    \end{align}
\end{subequations}

Assume that $[f(x_K),f(x_0))$ excludes the elements of $V$ and the averages of any two consecutive elements of $V$. Since we are excluding a finite number of elements, there exists $\epsilon>0$ such that $[f(x_K)-\epsilon,f(x_0))$ satisfies the same property. If $x_0$ is a local minimum of $f$ on the open set $X$, then $\nabla f(x_0) = 0$ by Fermat's rule. The length of $x_0,\hdots,x_{K+1}$ is then equal to zero and the formula obtained in \eqref{eq:length_finite} trivially holds. Otherwise, there exists a sequence $(x_0^l)_{l\in\mathbb{N}}$ in $\mathbb{R}^n$ such that $f(x_0^l) < f(x_0)$ and $x_0^l$ converges to $x_0$. For every $l \in \mathbb{N}$, consider the sequence $(x^l_k)_{k\in \mathbb{N}}$ such that $x^l_{k+1} = x^l_k - \alpha_k \nabla f(x_k^l)$ for all $k \in \mathbb{N}$. Since $\nabla f$ is continuous, for any fixed $k \in \{0,\hdots,K\}$, $x_k^l$ converges to $x_k \in X$ and eventually belongs to the open set $X$. Thus $[f(x^l_K),f(x^l_0)] \subset [f(x_K)-\epsilon,f(x_0))$ eventually excludes the elements of $V$ and the averages of any two consecutive elements of $V$. We can apply the formula obtained in \eqref{eq:length_finite}, namely
\begin{equation*}
    \sum_{k=0}^K \|x^l_{k+1}-x^l_k\| \leqslant (1+\epsilon/2) \psi( f(x^l_0) - f(x^l_K)) + L\max_{k=0,\hdots,K} \alpha_k.
\end{equation*}
Passing to the limit yields
\begin{equation*}
    \sum_{k=0}^K \|x_{k+1}-x_k\| \leqslant (1+\epsilon/2) \psi( f(x_0) - f(x_K)) + L\max_{k=0,\hdots,K} \alpha_k.
\end{equation*}

We next consider the general case where
\begin{equation*}
    [f(x_K),f(x_{K_p+1})) \cup \cdots \cup [f(x_{K_2}),f(x_{K_1+1})) \cup [f(x_{K_1}),f(x_0))
\end{equation*}
excludes the elements of $V$ and the averages of any two consecutive elements of $V$. For notational convenience, let $K_0 := -1$ and $K_{p+1} := K$. Since $p \leqslant 2m-1$, we have
\begin{align*}
    \sum_{k=0}^K \|x_{k+1}-x_k\| = & \sum_{i=0}^p \sum_{k=K_i+1}^{K_{i+1}} \|x_{k+1}-x_k\| \\
    \leqslant & \sum_{i=0}^p \left((1+\epsilon/2)\psi( f(x_{{K_i}+1}) - f(x_{K_{i+1}})) + L\max_{k=0,\hdots,K} \alpha_k\right) \\[-.5mm]
    \leqslant & (1+\epsilon/2)(p+1) ~ \psi \left( \frac{1}{p+1} \sum_{i=0}^p  f(x_{{K_i}+1}) - f(x_{K_{i+1}}) \right) + \\ 
    & (p+1)L\max_{k=0,\hdots,K} \alpha_k \\
    \leqslant & (1+\epsilon/2)(p+1) \psi \left( \frac{f(x_0)-f(x_K)}{p+1} \right) + (p+1)L\max_{k=0,\hdots,K} \alpha_k \\
    \leqslant & (2+\epsilon)m ~ \psi \left( \frac{f(x_0)-f(x_K)}{2m} \right) + 2m L\max_{k=0,\hdots,K} \alpha_k.\qedhere
\end{align*}
\end{proof}

An immediate consequence of Propositions \ref{prop:length_continuous} and \ref{prop:length_discrete} is that bounded continuous subgradient trajectories and bounded discrete gradient trajectories have finite length, as is well known. An easy fact that we will not prove here is that continuous subgradient trajectories of finite length converge to critical points of the objective function under the sole assumption that it is locally Lipschitz. The same holds for discrete subgradient trajectories if the step sizes are not summable.

Propositions \ref{prop:length_continuous} and \ref{prop:length_discrete} shed light on a seemingly unexplored consequence of the Kurdyka-\L{}ojasiewicz inequality. When considering subgradient trajectories in a bounded region, length, not time nor the number of iterations, guarantees a uniform decrease of the objective function. Indeed, since $\psi$ is bijective and strictly increasing, one can compose with $\psi^{-1}$ while preserving the order of the inequalities in \eqref{eq:length_continuous} and \eqref{eq:length_discrete}. This is illustrated in Example \ref{eg:decrease}.
\begin{example}
\label{eg:decrease}
The function $f(x_1,x_2) := x_1^3 - x_1^2x_2$ in Figure \ref{fig:decrease}, whose critical points are denoted in black, admits the desingularizing function $\psi(t) := 3t^{1/3}$ on $\mathbb{R}^2$. Hence for all sufficiently small step sizes, discrete gradient trajectories initialized in $B(0,0.3)$ below the critical value decrease by at least $2\psi^{-1}(1/9) \geqslant 0.0001$ by the time they exit $B(0,0.8)$. $\psi$ is a desingularizing function of $f$ on $\mathbb{R}^2$ because $\|\nabla (\psi \circ | f| )(x)\| \geqslant 1$ for all $(x_1,x_2)\in \mathbb{R}^2$ such that $f(x_1,x_2)\neq 0$, i.e., $x_1 \neq 0$ and $x_1 \neq x_2$. The inequality follows from $\|\nabla f(x_1,x_2)\| = ((3x_1^2 - 2x_1x_2)^2 + x_1^4)^{1/2} \geqslant | x_1^3 - x_1^2x_2|^{2/3}$. Indeed, $((3x_1^2 - 2x_1x_2)^2 + x_1^4)^3 - (x_1^3 - x_1^2x_2)^4 = x_1^6 [((3x_1 - 2x_2)^2 + x_1^2)^3 - x_1^2(x_1 - x_2)^4] = x_1^6 [((x_1 +2x_3)^2 + x_1^2)^3 - x_1^2x_3^4]$ where $x_3 := x_1-x_2$. By the inequality of arithmetic and geometric means, for any $c \in (\sqrt{2}/2,1)$ we have
    \begin{align*}
        (x_1 +2x_3)^2 + x_1^2 & = 2(x_1^2+2x_1x_3+2x_3^2) \\
        & = 2[(cx_1+x_3/c)^2 + (1-c^2)x_1^2 + (2-1/c^2)x_3^2] \\
        & \geqslant 2(1-c^2)x_1^2 + (2-1/c^2)x_3^2 + (2-1/c^2)x_3^2 \\
        & \geqslant 3[2(1-c^2)x_1^2(2-1/c^2)^2x_3^4]^{1/3}.
    \end{align*}
Thus $((x_1 +2x_3)^2 + x_1^2)^3 \geqslant 54(1-c^2)(2-1/c^2)^2 x_1^2x_3^4 = 9x_1^2x_3^4/2 \geqslant x_1^2x_3^4$ by taking $c = \sqrt{2/3}$ in order to obtain the equality.
\begin{figure}[H]
    \centering
    \includegraphics[width=1\linewidth]{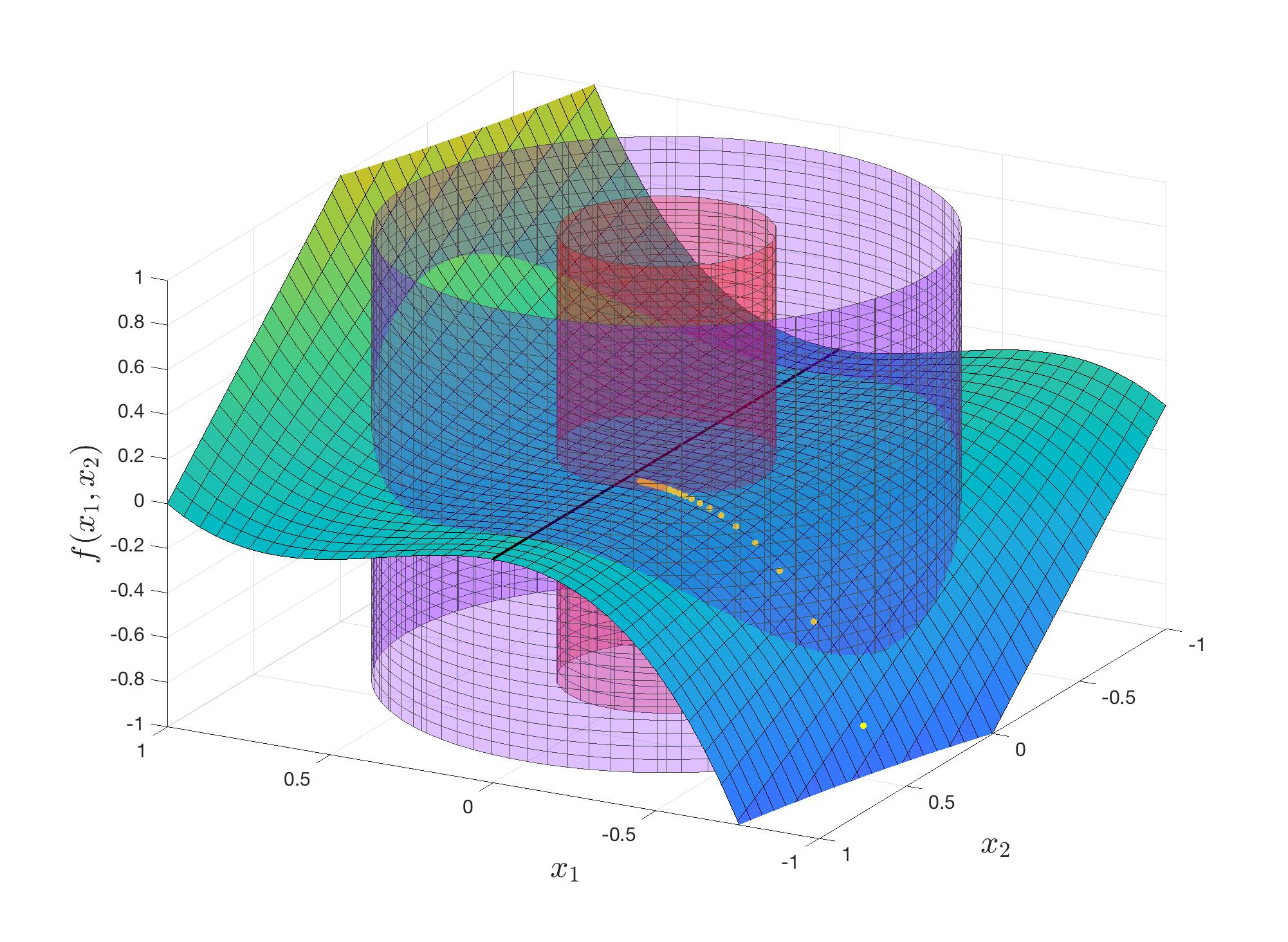}
    \caption{$f(x_1,x_2)=x_1^3 - x_1^2x_2$.}
    \label{fig:decrease}
\end{figure}
\end{example}

\section{Proof of Lemma \ref{lemma:uniform_length}}
\label{sec:Proof of Lemma lemma:uniform_length}
 
Let $f:\mathbb{R}^n \rightarrow \mathbb{R}$ be a locally Lipschitz definable function with bounded continuous subgradient trajectories. Let $X_0$ be a bounded subset of $\mathbb{R}^n$. We will show that $\sigma(X_0) < \infty$ where 
\begin{subequations}
\label{eq:sup_trajectory}
\begin{align}
    \sigma(X_0) := & \sup\limits_{x \in \mathcal{A}(\mathbb{R}_+,\mathbb{R}^n)} ~~ \int_0^{\infty} \|x'(t)\|dt \\
  & ~~ \mathrm{subject~to} ~~~ 
\left\{ 
\begin{array}{l}
x'(t) \in - \partial f(x(t)), ~\mathrm{for~a.e.}~ t > 0,\\[3mm] x(0) \in X_0.
\end{array}
\right.
\end{align}
\end{subequations}
Without loss of generality, we may assume that $X_0\neq \emptyset$. By Propositions \ref{prop:existence_trajectory} and \ref{prop:bounded_global}, the feasible set of \eqref{eq:sup_trajectory} is thus non-empty and $\sigma(X_0)>-\infty$.

We first consider the case with finite time horizon $T\geqslant 0$. By the definable Morse-Sard theorem \cite[Corollary 9]{bolte2007clarke}, $f$ has finitely many critical values. Notice that $\sigma_T(X_0) \leqslant \sqrt{T(\sup_{X_0} f - m(f))}$ where $m(f)$ is the smallest critical value of $f$ and 
\begin{subequations}
\label{eq:sup_trajectory_T}
\begin{align}
    \sigma_T(X_0) := & \sup\limits_{x \in \mathcal{A}(\mathbb{R}_+,\mathbb{R}^n)} ~~ \int_0^T \|x'(t)\|dt \\
  & ~~ \mathrm{subject~to} ~~~ 
\left\{ 
\begin{array}{l}
x'(t) \in - \partial f(x(t)), ~\mathrm{for~a.e.}~ t > 0,\\[3mm] x(0) \in X_0.
\end{array}
\right.
\end{align}
\end{subequations}
Indeed, by the Cauchy-Schwarz inequality and Propositions \ref{prop:chain_rule} and \ref{prop:length_continuous}, for any feasible point $x(\cdot)$ of \eqref{eq:sup_trajectory_T} we have
\begin{subequations}
\begin{align}
    \int_0^T \|x'(t)\|dt & \leqslant \sqrt{T} \sqrt{ \int_0^T \|x'(t)\|^2 dt} \\
    & \leqslant \sqrt{T} \sqrt{ \int_0^{\infty} \|x'(t)\|^2 dt} \\
    & = \sqrt{T} \sqrt{f(x(0)) - f\left(\lim_{t\rightarrow\infty} x(t)\right)} \\
    & \leqslant \sqrt{T\left(\sup_{X_0} f - m(f)\right)}. \label{eq:finite_time}
\end{align}
\end{subequations}

We next treat the case with infinite time horizon. Consider a sequence of feasible points $x_0(\cdot), x_1(\cdot), x_2(\cdot), \ldots$ of \eqref{eq:sup_trajectory}  such that $\int_0^{\infty} \|x_k'(t)\|dt$ converges to $\sigma(X_0)$. We proceed to show that the sequence is equicontinuous. Let $\epsilon>0$ and $t\geqslant 0$. Consider problem \eqref{eq:sup_trajectory_T} with finite time horizon $T:=t+\epsilon$. Since $x_0(\cdot),x_1(\cdot),x_2(\cdot),\ldots$ are feasible points of \eqref{eq:sup_trajectory_T}, for all $s \in [0,T]$ and $k=0,1,2,\ldots$ we have
\begin{subequations}
    \begin{align}
        \|x_k(s)\| & \leqslant \|x_k(s) - x_k(0) \| + \|x_k(0)\| \label{upper_0} \\ & \leqslant \int_0^s \|x_k'(\tau)\|d\tau + \|x_k(0)\| \label{upper_1} \\
        & \leqslant \sigma_T(X_0) + \sup_{x_0 \in X_0} \|x_0\|. \label{upper_2}
    \end{align}
\end{subequations}
All the trajectories $x_0(\cdot),x_1(\cdot),x_2(\cdot),\ldots$ hence belong to a common ball centered at the origin up to time $T$. Let $L>0$ be a Lipschitz constant of $f$ on that ball. As a result, for all $s \in [t-\delta,t+\delta]$ with $\delta:=\epsilon/(2L)$ and for all $k\in \mathbb{N}$, we have
\begin{equation*}
    \|x_k(t)-x_k(s)\| = \left\| \int_{s}^t x'_k(\tau)d\tau \right\| \leqslant \int_{t-\delta}^{t+\delta} \|x'_k(\tau)\| d\tau \leqslant 2\delta L = \epsilon,
\end{equation*}
It follows that $x_0(\cdot),x_1(\cdot),x_2(\cdot),\ldots$ is equicontinuous on $\mathbb{R}_+$. In addition, \eqref{upper_0}-\eqref{upper_2} imply that, for all $t\geqslant 0$, the sequence $x_0(t),x_1(t),x_2(t),\ldots$ is bounded. The Arzel\`a-Ascoli theorem \cite[Theorem 1 p. 13]{aubin1984differential} implies that there exists a subsequence (again denoted $x_k(\cdot)$) converging uniformly over compact intervals to a continuous function $x:\mathbb{R}_+\rightarrow\mathbb{R}^n$. 

We next show that $x(\cdot)$ is a continuous subgradient trajectory of $f$. Let $T\geqslant 0$. By virtue of \eqref{upper_0}-\eqref{upper_2} and local Lipschitz continuity of $f$, the restrictions of $x'_0(\cdot),x'_1(\cdot),x'_2(\cdot),\ldots$ to $[0,T]$ lie in a ball of the dual space of $L^1([0,T],\mathbb{R}^n)$, namely $L^{\infty}([0,T],\mathbb{R}^n)$. By the Banach-Alaoglu theorem \cite[Theorem 3 p. 13]{aubin1984differential}, there exists a subsequence (again denoted $x_k'(\cdot)$) that converges weakly\textsuperscript{$\ast$} to a function $y(\cdot)$ in $L^{\infty}([0,T],\mathbb{R}^n)$. Together with $L^{\infty}([0,T],\mathbb{R}^n)\subset L^1([0,T],\mathbb{R}^n)$, we find that $x_k'(\cdot)$ converges weakly to $y(\cdot)$ in $L^1([0,T],\mathbb{R}^n)$. Since $x_k(t)-x_k(s) = \int_s^t x_k'(\tau)d\tau$ for all $0\leqslant s \leqslant t \leqslant T$, we have $x(t)-x(s) = \int_s^t y(\tau)d\tau$. As a result, $x'(t) = y(t)$ for almost every $t\in (0,T)$. According to \cite[Convergence Theorem p. 60]{aubin1984differential}, it follows that $x'(t) \in -\partial f(x(t))$ for almost every $t\in(0,T)$. As $T \geqslant 0$ was arbitrary, we have $x'(t) \in -\partial f(x(t))$ for almost every $t>0$.


Since $f$ is definable and has bounded continuous subgradient trajectories, by Proposition \ref{prop:length_continuous} $x(\cdot)$ has finite length and converges to a critical point $x^*$ of $f$. Let $\epsilon>0$ and let $m \in \mathbb{N}^*$ be the number of critical values of $f$ in $B(\overline{x(\mathbb{R}_+)},\epsilon)$. Let $\psi$ be a desingularizing function of $f$ on $B(\overline{x(\mathbb{R}_+)},\epsilon)$. Since $f$ is continuous, there exists $\delta \in (0,\epsilon/2)$ such that 
\begin{equation}
\label{eq:f_bound_continuous}
    f(x) - f(x^*) \leqslant m~\psi^{-1}\left(\frac{\epsilon}{4m}\right),~~~ \forall x\in B(x^*,\delta).
\end{equation}

Let $t^*\geqslant 0$ be such that $\|x(t)-x^*\| \leqslant \delta/2$ for all $t\geqslant t^*$. By the uniform convergence of $x_k(\cdot)$, we have that $\|x_k(t)-x(t)\|\leqslant \delta/2$ for all $t\in [0,t^*]$ for all $k$ large enough. Hence $\|x_k(t^*) - x^*\|\leqslant \|x_k(t^*)-x(t^*)\|+\|x(t^*)-x^*\| \leqslant \delta/2 + \delta/2 = \delta$. If $T_k := \inf\{ t \geqslant t^* : x_k(t) \notin \mathring{B}(x^*,\epsilon)\}<\infty$, then
\begin{subequations}
    \begin{align}
        2m~\psi^{-1}\left(\frac{\epsilon}{4m}\right) & \leqslant 2m~ \psi^{-1}\left(\frac{\epsilon-\delta}{2m}\right) \label{eq:decrease_continuous_a} \\
        & \leqslant 2m~ \psi^{-1}\left(\frac{\|x_k(T_k)-x^*\| - \|x_k(t^*)-x^*\|}{2m}\right) \label{eq:decrease_continuous_b} \\
        & \leqslant 2m~ \psi^{-1}\left(\frac{\|x_k(T_k)-x_k(t^*)\|}{2m}\right) \label{eq:decrease_continuous_c} \\
        & \leqslant 2m~ \psi^{-1}\left(\frac{1}{2m}\int_{t^*}^{T_k}\|x_k'(t)\|dt \right) \label{eq:decrease_continuous_d} \\
        & \leqslant f(x_k(t^*))-f(x_k(T_k)) \label{eq:decrease_continuous_e} \\[3mm]
        & \leqslant m~\psi^{-1}\left(\frac{\epsilon}{4m}\right) + f(x^*) - f(x_k(T_k)). \label{eq:decrease_continuous_f}
    \end{align}
\end{subequations}
Above, \eqref{eq:decrease_continuous_a} through \eqref{eq:decrease_continuous_d} rely on the fact that $\psi^{-1}$ is increasing. \eqref{eq:decrease_continuous_a} is due to $\delta < \epsilon/2$. \eqref{eq:decrease_continuous_b} holds because $\|x_k(T_k)-x^*\|\geqslant \epsilon$ by continuity of $x_k(\cdot)$ and $x_k(t^*) \in B(x^*,\delta)$. \eqref{eq:decrease_continuous_c} is a consequence of the triangular inequality. \eqref{eq:decrease_continuous_d} holds because $x_k(\cdot)$ is absolutely continuous. \eqref{eq:decrease_continuous_e} is due to Proposition \ref{prop:length_continuous} and the fact that $x_k(t) \in B(x^*,\epsilon)$ for all $t \in [t^*,T_k]$ by continuity of $x_k(\cdot)$. Finally, \eqref{eq:decrease_continuous_f} is due to $x_k(t^*)\in B(x^*,\delta)$ and \eqref{eq:f_bound_continuous}. Since $x_k(t) \in B(\overline{x(\mathbb{R}_+)},\epsilon)$ for all $t\in [0,T_k]$ and $x_k(T_k)$ belongs to
\begin{equation*}
     X_1 := B(x^*,\epsilon) \bigcap \left\{ x \in \mathbb{R}^n : f(x) \leqslant f(x^*) - m~\psi^{-1}\left(\frac{\epsilon}{4m}\right) \right\},
\end{equation*}
by Proposition \ref{prop:length_continuous} and the definition of $\sigma(\cdot)$ in \eqref{eq:sup_trajectory} we have 
    \begin{align*}
            \int_0^{\infty} \|x_k'(t)\|dt & = \int_{0}^{T_k} \|x_k'(t)\|dt + \int_{T_k}^{\infty} \|x_k'(t)\|dt \\
            & \leqslant 2m~\psi\left( \frac{1}{2m} \left(\sup_{X_0} f - \min_{B(x^*,\epsilon)} f\right) \right) + \max\{0,\sigma(X_1)\}.
    \end{align*}
Note that the inequality still holds if $T_k = \infty$. By taking the limit, we get
\begin{equation*}
    \sigma(X_0) ~\leqslant~ 2m~\psi\left( \frac{1}{2m} \left(\sup_{X_0} f - \min_{B(x^*,\epsilon)} f\right) \right) + \max\{0,\sigma(X_1)\}.
\end{equation*}

It now suffices to replace $X_0$ by $X_1$ and repeat the proof starting below \eqref{eq:finite_time}. A maximizing sequence $\bar{x}_k(\cdot)$ corresponding to $\sigma(X_1)$ converges to a continuous subgradient trajectory $\bar{x}(\cdot)$ whose initial point lies in the compact set $X_1$. If $X_1 \neq \emptyset$, then the critical value $f(\lim_{t\rightarrow\infty} \bar{x}(t))$ is less than or equal to $f(x^*) - m\psi^{-1}(\epsilon/(4m)) < f(x^*)$. By the definable Morse-Sard theorem \cite[Corollary 9]{bolte2007clarke}, $f$ has finitely many critical values. Thus, it is eventually the case that one of the sets $X_0,X_1,\hdots$ is empty. We conclude that $\sigma(X_0)<\infty$ by the above recursive formula. 

\section{Proof of Theorem \ref{thm:convergence}}
\label{sec:Proof of Theorem thm:convergence}

($1\Longrightarrow 2$) We prove the contrapositive. Assume that there exists $0< T \leqslant \infty$ and a differentiable function $x:[0,T)\rightarrow\mathbb{R}^n$ such that
\begin{equation*}
    x'(t) = -\nabla f(x(t)), ~~~ \forall t \in (0,T),
\end{equation*}
for which, for all $c>0$, there exists $t\in [0,T)$ such that $\|x(t)\| > c$. Let $x_0 := x(0)$. Let $\bar{\alpha}$ and $c$ be some positive constants. Let $\tilde{t} \in (0,T)$ be such that $\|x(\tilde{t})\| \geqslant c+2$. By Proposition \ref{prop:tracking} and Remark \ref{rem:tracking}, there exists $\tilde{\alpha}\in (0,\bar{\alpha}]$ such that for all $\alpha \in (0,\tilde{\alpha}]$, the sequence $x_0,x_1,x_2,\hdots \in \mathbb{R}^n$ defined by $x_{k+1} = x_k -\alpha \nabla f(x_k)$ for all $k \in \mathbb{N}$ satisfies $\|x_k-x(k\alpha)\| \leqslant 1/2$ for all $k\in \{0,\hdots, \lfloor \tilde{t}/\alpha \rfloor \}$. We use the notation $\lfloor t \rfloor$ to denote the floor of a real number $t$ which is the unique integer such that $\lfloor t \rfloor \leqslant t < \lfloor t \rfloor + 1$. If $\tilde{k} = \lfloor \tilde{t}/\alpha \rfloor$, then $\|x_{\tilde{k}} - x(\tilde{t})\| \leqslant \|x_{\tilde{k}}-x(\tilde{k}\alpha)\|+\|x(\tilde{k}\alpha)-x(\tilde{t})\| \leqslant 1/2 + \tilde{L}| \tilde{k}\alpha-\tilde{t}| \leqslant 1/2 + \tilde{L}\alpha$ where $\tilde{L}>0$ is a Lipschitz constant of $f$ on the convex hull of $x([\tilde{t}-\tilde{\alpha},\tilde{t}])$. Let $\alpha_0,\alpha_1,\alpha_2,\hdots \in (0,\bar{\alpha}]$ be the constant sequence equal to $\min \{\tilde{\alpha}, 1/(2\tilde{L})\}$ and consider the sequence $x_0,x_1,x_2,\hdots \in \mathbb{R}^n$ defined by $x_{k+1} = x_k -\alpha_k \nabla f(x_k)$ for all $k \in \mathbb{N}$. By the triangular inequality, we have $\|x_{\tilde{k}}\| \geqslant \|x(\tilde{t})\| - \|x_{\tilde{k}}-x(\tilde{t})\| \geqslant c+2 - 1/2 - 1/2>c$ where $\tilde{k} := \lfloor \tilde{t}/\alpha \rfloor$.

($2 \Longrightarrow 3$) Let $X_0$ be a bounded subset of $\mathbb{R}^n$. We will show that there exists $\bar{\alpha}>0$ such that $\sigma(X_0,\bar{\alpha})<\infty$ where
\begin{subequations}
\label{eq:sup_discrete_X_0}
\begin{align}
    \sigma(X_0,\bar{\alpha}) := & \sup\limits_{\tiny \begin{array}{c}x \in (\mathbb{R}^n)^{\mathbb{N}}\\ \alpha \in (0,\bar{\alpha}]^\mathbb{N} \end{array}} ~~ \sum\limits_{k=0}^{\infty} \|x_{k+1}-x_k\| \\
   & ~~ \mathrm{subject~to} ~~~ 
\left\{ 
\begin{array}{l}
x_{k+1} = x_k - \alpha_k \nabla f(x_k), ~ \forall k\in \mathbb{N},\\[2mm] x_0 \in X_0.
\end{array}
\right.
\end{align}
\end{subequations}

Let $\Phi:\mathbb{R}_+ \times \mathbb{R}^n \rightarrow \mathbb{R}^n$ be the gradient flow of $f$ defined for all $(t,x_0) \in \mathbb{R}_+ \times \mathbb{R}^n$ by $\Phi(t,x_0) := x(t)$ where $x(\cdot)$ is the unique continuous gradient trajectory of $f$ initialized at $x_0$. Uniqueness follows from the Picard–Lindel{\"o}f theorem \cite[Theorem 3.1 p. 12]{coddington1955theory}. Let $C$ be the set of critical points of $f$ in $\overline{\Phi(\mathbb{R}_+,X_0)}$. If $X_0 \subset C$, then we have $\sigma(X_0,\bar{\alpha}) \leqslant 0$ for all $\bar{\alpha}>0$ and there is nothing left to prove. Otherwise, since $C$ is closed by \cite[2.1.5 Proposition p. 29]{clarke1990}, there exists $\epsilon>0$ be such that $X_0\setminus \mathring{B}(C,\epsilon/6) \neq \emptyset$ where $\mathring{B}(C,\epsilon/6) := C + \mathring{B}(0,\epsilon/6)$. It is safe to assume this from now on.

Let $\psi$ be a desingularizing function of $f$ on $\mathring{B}(\overline{\Phi}_0,\epsilon)$, where we use the shorthand $\Phi_0 := \Phi(\mathbb{R}_+,X_0)$. By the definable Morse-Sard theorem \cite[Corollary 9]{bolte2007clarke}, $f$ has finitely many critical values. Let $m \in \mathbb{N}$ be the number of critical values of $f$ in $B(\overline{\Phi}_0,\epsilon)$. Note that $m \geqslant 1$ due to Proposition \ref{prop:length_continuous} and the fact that $X_0 \neq \emptyset$. Since $f$ is continuous and $C$ is compact by Lemma \ref{lemma:uniform_length}, there exists $\delta \in (0,\epsilon/2)$ such that  
\begin{equation}
    \label{eq:f_bound_discrete}
    f(x) - \max_{C} f \leqslant m~ \psi^{-1}\left(\frac{\epsilon}{40m}\right),~~~ \forall x\in B(C,\delta).
\end{equation}
Note that $\mathring{B}(\overline{\Phi}_0,\epsilon)$ is bounded due to Lemma \ref{lemma:uniform_length}. Let $L>0$ be a Lipschitz constant of $f$ on the convex hull of $\mathring{B}(\overline{\Phi}_0,\epsilon)$. By Proposition \ref{prop:length_discrete}, there exists $\alpha_\delta \in (0,\delta/(5mL)]$ such that, for all $K \in \mathbb{N}^*$, $\alpha_0,\hdots,\alpha_{K-1} \in (0,\alpha_\delta]$, and $(x_0,\hdots,x_K) \in \mathring{B}(\overline{\Phi}_0,\epsilon) \times \cdots \times \mathring{B}(\overline{\Phi}_0,\epsilon) \times \mathbb{R}^n$ such that $x_{k+1} = x_k - \alpha_k \nabla f(x_k)$ for $k = 0,\hdots,K-1$, we have
\begin{equation}
\label{eq:length_discrete_K}
    \frac{1}{4m}\sum_{k=0}^{K-1}\|x_{k+1}-x_k\| ~\leqslant~ \psi\left(\frac{f(x_0)-f(x_{K-1})}{2m}\right) + \frac{L}{2}\max_{k=0,\hdots,K-1} \alpha_k.
\end{equation}

Since $\nabla f$ is continuous, its norm attains its infimum $\nu$ on the nonempty compact set $\overline{\Phi}_0 \setminus \mathring{B}(C,\delta/3)$. 
It is nonempty because $\overline{\Phi}_0 \setminus \mathring{B}(C,\delta/3) \supset X_0\setminus \mathring{B}(C,\epsilon/6) \neq \emptyset$ and $\delta < \epsilon/2$.
If $\nu=0$, then there exists $x^* \in \overline{\Phi}_0 \setminus \mathring{B}(C,\delta/3)$ such that $\|\nabla f(x^*)\| = 0$. Then $x^*\in C \setminus \mathring{B}(C,\delta/3)$, which is a contradiction. We thus have $\nu>0$. Hence we may define $T := 2\sigma(X_0)/\nu$ where $\sigma(X_0)$ is defined in \eqref{eq:sup_trajectory} and is finite by Lemma \ref{lemma:uniform_length}. Note that $\sigma(X_0)>0$ and thus $T>0$ because $X_0\not\subset C$. By Proposition \ref{prop:tracking}, there exists $\bar{\alpha} \in (0,\alpha_\delta]$ such that for any feasible point $((x_k)_{k\in \mathbb{N}} , (\alpha_k)_{k\in\mathbb{N}})$ of \eqref{eq:sup_discrete_X_0},
the continuous gradient trajectory $x:\mathbb{R}_+\rightarrow \mathbb{R}^n$ of $f$ initialized at $x_0$ satisfies
\begin{equation}
\label{eq:tracking_T}
   \forall k\in \mathbb{N}^*,~~~ \alpha_0+\cdots+\alpha_{k-1} \leqslant T ~~~\Longrightarrow~~~ \|x_k - x(\alpha_0+\cdots+\alpha_{k-1})\| \leqslant \frac{\delta}{3}.
\end{equation}
Now suppose that $\|x'(t)\| \geqslant 2\sigma(X_0)/T$ for all $t\in (0,T)$. Then we obtain the following contradiction
\begin{equation*}
     \sigma(X_0) < T \frac{2\sigma(X_0)}{T} \leqslant \int_0^{T} \|x'(t)\|dt \leqslant \int_0^{\infty} \|x'(t)\|dt \leqslant \sigma(X_0).
\end{equation*}
Hence, there exists $t^* \in (0,T)$ such that $\|x'(t^*)\| = \|\nabla f(x(t^*))\| < 2\sigma(X_0)/T = 2\sigma(X_0)/(2\sigma(X_0)/\nu) = \nu$. Since  $x(t^*) \in \Phi(\mathbb{R}_+,X_0) \subset \overline{\Phi}_0$ and the infimum of the norm of $\nabla f$ on $\overline{\Phi}_0 \setminus \mathring{B}(C,\delta/3)$ is equal to $\nu$, it must be that $x(t^*) \in \mathring{B}(C,\delta/3)$. Hence there exists $x^* \in C$ such that $\|x(t^*)-x^*\|\leqslant \delta/3$. 

If $\sum_{k=0}^{\infty} \alpha_k \leqslant T$, then by \eqref{eq:tracking_T} we have $x_k \in B(\overline{\Phi}_0,\delta/3)$ for all $k \in \mathbb{N}$. Otherwise, since $\alpha_k \leqslant \bar{\alpha} \leqslant \alpha_\delta \leqslant \delta/(5mL) \leqslant \delta/(3L)$ for all $k \in \mathbb{N}$, there exists $k^* \in \mathbb{N}$ such that $t_{k^*} := \sum_{k=0}^{k^*-1} \alpha_k \in [ t^* - \delta/(3L) , t^*]$, where $\sum_{k=0}^{k^*-1} \alpha_k = 0$ if $k^* = 0$. Thus $\|x_{k^*}-x^*\| \leqslant \|x_{k^*}-x(t_{k^*})\| + \|x(t_{k^*}) - x(t^*)\| + \|x(t^*)-x^*\| \leqslant \delta/3 + L | t_{k^*}-t^* | + \delta/3 \leqslant \delta/3+\delta/3+\delta/3 = \delta$.
If $K := \inf\{ k \geqslant k^* : x_k \notin \mathring{B}(C,\epsilon)\}<\infty$, then
\begin{subequations}
    \begin{align}
        2m~\psi^{-1}\left(\frac{\epsilon}{40m}\right) & = 2m~ \psi^{-1}\left(\frac{1}{8m} \epsilon - \frac{L}{2}\frac{\epsilon}{5mL} \right) \label{eq:decrease_discrete_a} \\
        & \leqslant 2m~ \psi^{-1}\left(\frac{1}{4m}(\epsilon-\delta) - \frac{L\alpha_\delta}{2} \right) \label{eq:decrease_discrete_b} \\
        & \leqslant 2m~ \psi^{-1}\left(\frac{1}{4m} \left(\|x_K-x^*\|-\|x_{k^*}-x^*\|\right) - \frac{L\alpha_\delta}{2} \right) \label{eq:decrease_discrete_c} \\
        & \leqslant 2m~ \psi^{-1}\left(\frac{1}{4m} \|x_K-x_{k^*}\| - \frac{L\bar{\alpha}}{2} \right) \label{eq:decrease_discrete_d} \\
        & \leqslant 2m~ \psi^{-1}\left(\frac{1}{4m}\sum_{k=k^*}^{K-1} \|x_{k+1}-x_k\| - \frac{L}{2}\max_{k=0,\hdots,K-1} \alpha_k \right) \label{eq:decrease_discrete_e} \\
        & \leqslant f(x_{k^*}) - f(x_{K-1}) \label{eq:decrease_discrete_f} \\
        & \leqslant m~\psi^{-1}\left(\frac{\epsilon}{40m}\right) + \max_C f - f(x_{K-1}). \label{eq:decrease_discrete_g}
    \end{align}
\end{subequations}
Above, the arguments of $\psi^{-1}$ in \eqref{eq:decrease_discrete_a} are equal. \eqref{eq:decrease_discrete_b} through \eqref{eq:decrease_discrete_e} rely on the fact that $\psi^{-1}$ is an increasing function. \eqref{eq:decrease_discrete_b} is due to $\delta<\epsilon/2$. \eqref{eq:decrease_discrete_c} holds because $x_K \notin \mathring{B}(C,\epsilon)$, $x^* \in C$, and $x_{k^*} \in B(x^*,\delta)$. \eqref{eq:decrease_discrete_d} and \eqref{eq:decrease_discrete_e} are consequences of the triangular inequality. \eqref{eq:decrease_discrete_f} is due to the length formula \eqref{eq:length_discrete_K} and the fact that $x_{k^*},\hdots,x_{K-1} \in \mathring{B}(C,\epsilon) \subset \mathring{B}(\overline{\Phi}_0,\epsilon)$. Finally, \eqref{eq:decrease_discrete_g} is due to $x_{k^*}\in B(C,\delta)$ and \eqref{eq:f_bound_discrete}. We remark that $K\geqslant k^*+2$ since $\|x_{k^*+1}-x^*\| \leqslant \|x_{k^*+1}-x_{k^*}\| +\|x_{k^*}-x^*\| \leqslant \alpha_{k^*}\|\nabla f(x_{k^*})\|+\delta \leqslant (\delta/(5mL))L+\delta < \epsilon$. Since $x_0,\hdots,x_{K-2} \in \mathring{B}(\overline{\Phi}_0,\epsilon)$ and $x_{K-1}$ belongs to 
\begin{equation*}
    X_1 := B(C,\epsilon) \bigcap \left\{ x \in \mathbb{R}^n : f(x) \leqslant \max_{C} f - m~\psi^{-1}\left(\frac{\epsilon}{40m}\right) \right\},    
\end{equation*}
by the length formula \eqref{eq:length_discrete_K} and the definition of $\sigma(\cdot,\cdot)$ in \eqref{eq:sup_discrete_X_0} we have
    \begin{align*}
            \sum_{k=0}^\infty \|x_{k+1}-x_k\| & = \sum_{k=0}^{K-2} \|x_{k+1}-x_k\| + \sum_{k=K-1}^\infty \|x_{k+1}-x_k\| \\
            & \leqslant 4m~\psi\left( \frac{1}{2m} \left(\sup_{X_0} f - \min_{B(\overline{\Phi}_0,\epsilon)} f\right) \right) + 2mL\bar{\alpha} + \max\{0,\sigma(X_1,\bar{\alpha})\}.
    \end{align*}
Note that the inequality still holds if $K = \infty$ or $\sum_{k=0}^{\infty} \alpha_k \leqslant T$. Hence
\begin{equation*}
    \sigma(X_0,\bar{\alpha}) ~\leqslant~ 4m~\psi\left( \frac{1}{2m} \left(\sup_{X_0} f - \min_{B(\overline{\Phi}_0,\epsilon)} f\right) \right) + 2mL\bar{\alpha} + \max\{0,\sigma(X_1,\bar{\alpha})\}.
\end{equation*}

It now suffices to replace $X_0$ by $X_1$ and repeat the proof starting below \eqref{eq:sup_discrete_X_0}. Since $f(\Phi(t,x_1)) \leqslant f(\Phi(0,x_1)) \leqslant \max_{C} f - m\psi^{-1}(\epsilon/(40m))< \max_{C} f$ for all $t\geqslant 0$ and $x_1\in X_1$, the maximal critical value of $f$ in $\overline{\Phi(\mathbb{R}_+,X_1)}$ is less than the maximal critical value of $f$ in $\overline{\Phi(\mathbb{R}_+,X_0)}$. By the definable Morse-Sard theorem \cite[Corollary 9]{bolte2007clarke}, $f$ has finitely many critical values. Thus, it is eventually the case that one of the sets $X_0,X_1,\hdots$ is empty. In order to conclude, one simply needs to choose an upper bound on the step sizes $\hat{\alpha}$ corresponding to $X_1$ that is less than or equal to the upper bound $\bar{\alpha}$ used for $X_0$. $\sigma(X_0,\cdot)$ is finite when evaluated at the last upper bound thus obtained. Indeed, the above recursive formula still holds if we replace $\bar{\alpha}$ by any $\alpha \in (0,\bar{\alpha}]$. In particular, we may take $\alpha := \hat{\alpha}$.

($3 \Longrightarrow 1$) Obvious.\\

\noindent\textbf{Acknowledgements} I am grateful to the reviewers and editors for their precious time and valuable feedback. Many thanks to Lexiao Lai and Xiaopeng Li for fruitful discussions. 

\bibliographystyle{abbrv}    
\bibliography{mybib}

\end{document}